\documentclass[graybox]{svmult}

\usepackage{type1cm}        
\usepackage{makeidx}         
\usepackage{graphicx}        
\usepackage{multicol}        
\usepackage[bottom]{footmisc}
\usepackage{verbatim}

\usepackage{newtxtext}       %
\usepackage[varvw]{newtxmath}       
\usepackage{overpic}

\makeindex             

\newtheorem{defn0}{Definition}[subsection]

\theoremstyle{definition}
\newtheorem{example0}[definition]{Example}

\newcommand{\NN}{{\mathbb N}}

\newcommand{\ZZ}{{\mathbb Z}}
\newcommand{\RR}{{\mathbb R}}

\newcommand{\CC}{{\mathbb C}}

\newcommand{\base}{\mathbf{a}}

\newcommand{\w}{\mathbf{w}}
\def\inv{\operatorname{Inv}}

\newcommand{\type}{\operatorname{type}}

\definecolor{limegreen}{rgb}{0.2, 0.8, 0.2}
\definecolor{amethyst}{rgb}{0.6, 0.4, 0.8}
\definecolor{deeppink}{rgb}{1.0, 0.08, 0.58}

\newcommand{\red}{\textcolor{red}}

\newcommand{\blue}{\textcolor{blue}}

\begin{document}

\title*{Alcove Walks and GKM Theory for Affine Flags}
\author{Elizabeth Mili\'cevi\'c and Kaisa Taipale}
\institute{Elizabeth Mili\'cevi\'c, \email{emilicevic@haverford.edu}
\and Kaisa Taipale, \email{taipale@umn.edu}}
%
%
\maketitle

\abstract{We develop the GKM theory for the torus-equivariant cohomology of the affine flag variety using the combinatorics of alcove walks. Dual to the usual GKM setup, which depicts the orbits of the small torus action on a graph, alcove walks take place in tessellations of Euclidean space. Walks in affine rank two occur on triangulations of the plane, providing a more direct connection to splines used for approximating surfaces. Alcove walks in GKM theory also need not be minimal length, and can instead be randomly generated, giving rise to more flexible implementation. This work reinterprets and recovers classical results in GKM theory on the affine flag variety, generalizing them to both non-minimal and folded alcove walks, all motivated by applications to splines.}

\section{Introduction}
\label{sec:intro}

GKM theory is a powerful method for doing cohomological calculations on a vast array of topological spaces, by recording a strikingly small amount of information.  When a group such as an algebraic torus acts suitably on a space like an algebraic variety, the fixed points and one-dimensional orbits become the vertices and edges of a collection of graphs which fully encode the structure of the corresponding equivariant cohomology ring.  This  approach is due to Goresky, Kottwitz, and MacPherson for certain algebraic varieties \cite{GKM}, was developed in the Kac-Moody setting by Kostant and Kumar \cite{KostantKumar}, and further generalized by Henriques, Harada, and Holm \cite{HHH}; see also the survey on GKM theory by Tymoczko \cite{TymGKM}.

\begin{figure}
	\centering
	\includegraphics[scale=0.23]{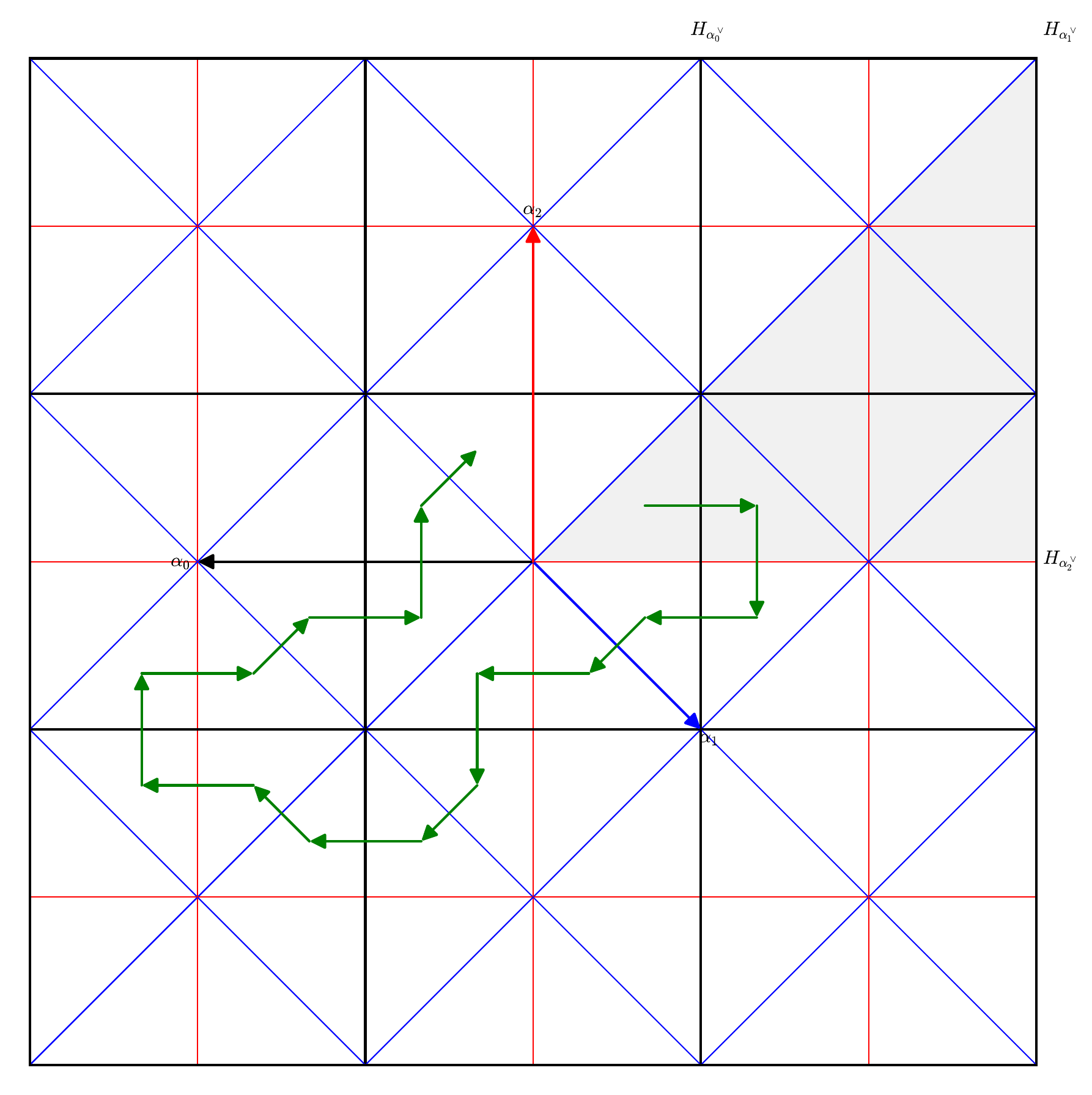}
	\caption{An alcove walk in affine type $C_2$ with 16 steps shown in green, not minimal length.}
	\label{fig:C2walk}
\end{figure}

Exciting applications of GKM theory to approximation theory have been emerging, pioneered by Gilbert, Tymoczko, and Viel \cite{gilbert2016generalized}, building on foundational work linking commutative algebra to multivariate splines by Billera and Rose in the simplicial case \cite{billera1989grobner}, with refinements by Schenck and Stillman \cite{ schenck1997spectral, schenck1997local}.  The article \cite{gilbert2016generalized} also lays out the extensive intellectual history of the interplay among GKM theory, approximation theory, splines, algebra, and combinatorics.

Splines provide piecewise polynomial approximations for smooth surfaces in 3D space, supported over a triangular or quadrilateral mesh; see the textbook by Lai and Schumaker for an introduction \cite{lai2007spline}. 
The methods of GKM theory are typically represented on the dual spaces to triangulations, however, since passing to the cohomological framework reverses high and low-dimensional features.

While there is a burgeoning literature on connections between GKM theory and the resulting topological approach to splines theory (e.g.~rediscovering splines as equivariant cohomology rings of toric varieties), there is currently no literature linking alcove walks to splines.  This paper reinterprets GKM theory in terms of alcove walks like the one depicted in Figure \ref{fig:C2walk}, preserving the picture of triangulations of the plane as describing certain smooth surfaces supported above these alcoves.  This geometric description, exploiting the triangulations of the plane naturally arising in the GKM approach to affine flag varieties, can thus perhaps be a more helpful representation for certain splines in approximation theory.

\subsection{Summary of results}

We begin in Section \ref{sec:alcoves} by reviewing the affine hyperplane arrangement, whose complement is a collection of simplices called alcoves, which tessellate Euclidean space.  We represent alcove walks as a sequence of concatenated arrows between pairs of adjacent alcoves, as illustrated in green in Figure \ref{fig:C2walk}.  Each step crosses exactly one hyperplane, which we index by the corresponding affine root, signed according to whether the step moves closer or farther from the start.  

Each alcove corresponds to a unique element of the affine Weyl group $W_{af}$.  Any alcove walk $\gamma$ is thus naturally associated to the element $w \in W_{af}$ indexing the final alcove $\w$, as well as a choice of (not necessarily reduced) expression for $w$ in terms of the Coxeter generators labeling each step.  A mask $\varepsilon$ on an alcove walk $\gamma$ is a binary vector with the same number of entries as steps in $\gamma$.  Each entry of $\varepsilon$ is viewed as showing or hiding the corresponding step in $\gamma$, naturally forming an expression for an element $v \in W_{af}$ which is below $w$ in the affine Bruhat order.  

In Section \ref{sec:polys}, we introduce a new polynomial algebra on alcove walks, by assigning a particular product of affine roots to each mask. When summing over all masks for a given $v \leq w$, we show in Proposition \ref{prop:Indep} that the resulting polynomial is independent of the initial choice of alcove walk to $\w$, including that $\gamma$ need not be minimal.  This analysis gives rise to a family of functions $\psi^v$ whose value at $w$ records the hyperplane crossings shown by any mask for $v$ in a fixed walk to $\w$.

We then provide a brief review of GKM theory for affine flag varieties and the localization map in Section \ref{sec:GKMresults}.  The inclusion of the affine torus-fixed points into the affine flag variety induces an injection on the equivariant cohomology into a direct sum of polynomial rings.  Moreover, the image is characterized by a simple divisibility relation called the GKM condition.  In Theorem \ref{thm:GKM}, we prove that the functions $\psi^v$ satisfy the GKM condition, placing them in the image of this localization map.  Theorem \ref{thm:Local} then recognizes $\psi^v$ as the affine Schubert class $[X_v]$, identifying these alcove walk functions  with the most natural additive basis for the equivariant cohomology ring.  These results illustrate the advantage of working with alcove walks in GKM theory, since the walk simply ``knows'' which steps are extraneous, automatically recovering the classical localization formulas without an additional minimality hypothesis.

Alcove walks can also be folded according to rules governed by assigning an orientation to the affine hyperplane arrangement.  We develop these folded alcove walks in Section \ref{sec:folded}.  We prove a bijection in Theorem \ref{thm:AlcoveWalks} between positively folded alcove walks labeled by elements of $\CC$ and the intersection of two affine Schubert cells.  We conclude the paper by discussing the relationship between positively folded alcove walks and masks for alcove walks, equivalently summands of the function values $\psi^v(w)$, which govern the GKM theory of the affine flag variety.  Inspiring our use of the terminology of masks, the connection to Kazhdan-Lusztig polynomials can be made explicit by interpreting the recursion on $R$-polynomials as determined by folded alcove walks, labeled by elements of the finite field of order $q$.

\subsection*{Acknowledgements} 

This work was heavily influeced by countless related conversations with Arun Ram, Petra Schwer, and Anne Thomas, on many variations and generalizations of the alcove walks considered in this paper.  All of our figures were produced in Sage, and we thank the Sage combinatorics developers who have implemented useful features of alcove walks which we are pleased to illustrate \cite{SageMath, Sage-combinat}.  We also gratefully acknowledge the close reading of several anonymous referees who contributed to many improvements throughout the paper.

The authors  gratefully acknowledge a variety of funding sources and institutions, which provided excellent opportunities for the authors to discuss the project.
EM was partially supported by NSF Grants DMS-1600982 and DMS-2202017.  EM and KT gratefully acknowledge the support of the Max-Planck-Institut f\"ur Mathematik, which hosted two long-term sabbatical visits by EM, in addition to a collaborative visit by KT, jointly funded by EM's Simons Collaboration Grant  318716.  EM and KT also spent two weeks at the Institute for Advanced Study in the Summer Collaborators Program.

\section{Alcove Geometry}\label{sec:alcoves}
\label{sec:algebra}

In this section, we begin by reviewing all necessary terminology of affine Weyl groups and affine root systems.  We then develop the affine hyperplane arragnement and the machinery of alcove walks.  The notion of masks on expressions for affine Weyl group elements, equivalently on alcove walks, plays a central role.

\subsection{Affine Weyl groups and root systems}

We begin with a brief summary of some basic terminology and notation for affine Weyl groups and affine root systems.  For more details, we refer the reader to expository references such as \cite[Chapter 1]{Kumar} or \cite[Chapter 4]{kSchur}.

Let $(I_{af},A_{af})$ be an \emph{affine Cartan datum}, where $I_{af}$ is an indexing set such that $|I_{af}| = n$, and $A_{af} = (a_{ij})_{i,j \in I_{af}}$ is a generalized Cartan matrix, having 2's on the diagonal and otherwise nonpositive entries such that $a_{ij}\neq 0 \iff a_{ji}\neq 0$. Since $(I_{af},A_{af})$ is affine, then the matrix $A_{af}$ has corank 1, and the restriction of $A_{af}$ to $J \times J$ for any subset $J \subsetneq I_{af}$ yields a positive definite matrix.
For any $i,j \in I_{af}$, define an integer $m_{ij} = \{2,3,4,6\}$ if and only if $a_{ij}a_{ji} = \{0,1,2,3\}$, respectively, noting that these are the only possible values in the affine setting.

The \emph{affine Weyl group} associated to the affine Cartan datum $(I_{af},A_{af})$ is the group $W_{af}$ generated by $\{ s_i \mid i \in I_{af} \}$, subject to the relations $s_i^2 = 1$ and $(s_is_j)^{m_{ij}} = 1$.  The \emph{length} $\ell(w)$ of $w \in W_{af}$ is the minimum $m \in \NN \cup \{0\}$ such that $w = s_{i_1}s_{i_2} \cdots s_{i_m}$ for indices $i_j \in I_{af}$. In figures, we often use the abbreviated notation $w = s_{i_1i_2 \cdots i_m}$ for an \emph{expression} for $w$.  A factorization of $w \in W_{af}$ using exactly $\ell(w)$ simple reflections is called a \emph{reduced} expression for $w$.  The \emph{Bruhat order} on $W_{af}$ is defined by $v \leq w$ if and only if some (equivalently every) reduced expression for $w$ contains a subexpression which is a reduced expression for $v$. 

Fix a (not necessisarily reduced) expression $w = s_{i_1}s_{i_2} \cdots s_{i_m}$. Every binary vector $\varepsilon = (\varepsilon_1, \dots, \varepsilon_m) \in \{0,1\}^m$ then corresponds to a subexpression of $w$, which we denote by $w^\varepsilon = s_{i_1}^{\varepsilon_1}s_{i_2}^{\varepsilon_2} \cdots s_{i_m}^{\varepsilon_m}$. Borrowing terminology from Kazhdan-Lusztig polynomials \cite{BilleyWarrington}, we refer to such a vector $\varepsilon$ as a \emph{mask} associated to the expression $w = s_{i_1}\cdots s_{i_m}$, or simply a \emph{mask on $w$} when the expression is understood. If entry $\varepsilon_\ell = 1$, we say that the mask $\varepsilon$ \emph{shows} the generator $s_{i_\ell}$, and conversely if $\varepsilon_\ell = 0$, then $\varepsilon$ \emph{hides} $s_{i_\ell}$. The total number of 1's in a mask is called the \emph{support} of $\varepsilon$, which we denote by $|\varepsilon|$. Given $v \leq w$, we say that $\varepsilon$ is a \emph{mask for $v$} if $w^\varepsilon = v$. For any subinterval of $\{1, \dots, m\}$, say $J_k= \{j, j+1, \dots, j+k-1\}$ of length $k$ starting at $j$, we can also consider the \emph{submask} $\varepsilon_{J_k} = (\varepsilon_j, \dots, \varepsilon_{j+k-1}) \in \{0,1\}^k$, which corresponds to the subexpression $w^{\varepsilon_{J_k}}$ of $w$.  

 Denote by $(X_{af},X_{af}^{\vee}, \Delta_{af}, \Delta_{af}^\vee)$ the \emph{affine root datum} associated to $(I_{af},A_{af})$, where $X_{af}$ is a free $\ZZ$-module of rank $n$, and $X_{af}^\vee = \operatorname{Hom}(X_{af},\ZZ)$ is the dual lattice under the evaluation pairing $\langle \cdot , \cdot \rangle : X_{af}^{\vee} \times X_{af} \to \ZZ$.  The bases $\Delta_{af} = \{ \alpha_i \mid i \in I_{af} \}$ and $\Delta_{af}^\vee = \{ \alpha^\vee_i  \mid i \in I_{af}\}$ of positive simple affine roots and coroots, respectively, satisfy $\langle \alpha_i^\vee, \alpha_j \rangle = a_{ij}$ for any $i, j \in I_{af}$.  Fixing these bases, $X_{af} = \bigoplus_{i \in I_{af}} \ZZ \alpha_i$ is the \emph{affine root lattice} and $X_{af}^\vee =\bigoplus_{i \in I_{af}} \ZZ \alpha^\vee_i$ is the \emph{affine coroot lattice}. 
  The affine Weyl group acts on the affine root and coroot lattices by 
  \begin{equation}\label{eq:Waction}
  s_i \cdot \mu =  \mu - \langle \alpha^\vee_i, \mu \rangle \alpha_i  \quad\  \text{and} \quad\ 
  s_i \cdot \lambda = \lambda - \langle \lambda, \alpha_i \rangle \alpha_i^\vee
  \end{equation}
  for any $\mu \in X_{af}$ and $\lambda \in X_{af}^\vee$.  Restricting the action of $W_{af}$ on $X_{af}^\vee$ to the set $\Delta_{af}^\vee$, we obtain the set of real \emph{affine roots} $R_{af} = \{ w\cdot \alpha_i \mid w \in W_{af}, \; i \in I_{af}\}$.  The \emph{positive affine roots} are then $R_{af}^+ = R_{af} \cap \bigoplus_{i \in I_{af}} \ZZ_{\geq 0} \alpha_i$, and we write $R_{af} = R^+_{af} \sqcup R^-_{af}$.  Therefore, given any $\beta \in R_{af}$, either $\beta \in R^+_{af}$ or $-\beta \in R^+_{af}$, so we denote by $\beta^+$ the unique root in $\{\pm \beta\}$ which lies in $R^+_{af}$, and similarly by $\beta^-$ the one of $\pm \beta$ in $R^-_{af}$.

 Let $(I_0,A_0)$ be the finite Cartan datum obtained by restricting $A_{af}$ to $I_0\times I_0$, where $I_0 = I_{af}-\{0\}$. The finite Weyl group $W$ associated to $(I_0,A_0)$ is the subgroup of $W_{af}$ generated by $\{s_i \mid i \in I_0 \}$, subject to the same relations. The associated finite root datum $(X,X^\vee, \Delta, \Delta^\vee)$ consists of positive simple roots $\Delta = \Delta_{af}-\{\alpha_0\}$ and coroots $\Delta^\vee = \Delta^\vee_{af}-\{\alpha^\vee_0\}$, with root and coroot lattices $X =\bigoplus_{i \in I_0} \ZZ \alpha_i$ and $X^\vee =\bigoplus_{i \in I_0} \ZZ \alpha^\vee_i$, respectively.  The action of $W$ on $X$ and $X^\vee$ is also defined by  \eqref{eq:Waction}. The set of finite roots is given by $R = \{ w\cdot \alpha_i \mid w \in W, i \in I_0\}$. The positive finite roots are $R^+ = R \cap \bigoplus_{i \in I_0} \ZZ_{\geq 0} \alpha_i$, and we write $R = R^+ \sqcup R^-$. Let $\theta$ denote the highest root of $R^+$.
 Defining the \emph{null root} $\delta = \alpha_0 + \theta$, we then have
\begin{align*}
R_{af} &= \{ \alpha + k \delta \mid \alpha \in R,\ k \in\ZZ\} \nonumber \\
R^+_{af} &=  \{ \alpha + k \delta \mid \alpha \in R,\ k \in\ZZ_{>0}\} \cup R^+ \\
R^-_{af} & = R_{af} \backslash R^+_{af}. \nonumber
\end{align*}
We will usually express affine roots in this way, as a sum of a finite root $\alpha \in R$ and an \emph{imaginary root} of the form $k \delta$.

\subsection{Hyperplanes and alcoves}

We continue by defining a hyperplane arrangement associated to the affine root system, and we review the terminology of alcove walks introduced in \cite{Ram}. 

 Denote by $V^* = X^\vee \otimes{_\ZZ} \RR \cong \RR^{n-1}$. By definition, the generator $s_i\in W$ fixes the hyperplane $H_{\alpha_i} = \{ v \in V^* \mid \langle v, \alpha_i \rangle = 0 \}$, and the origin in $V^*$ is the intersection $\cap_{i\in I_0} H_{\alpha_i}$.  More generally, given any finite root $\alpha \in R$ and any integer $k \in \ZZ$, define a \emph{hyperplane} in $V^*$ by
$H_{\alpha+k\delta} = \{ v \in V^* \mid \langle v, \alpha \rangle = k\}.$
The hyperplanes for affine type $A_2$ are labeled in Figure \ref{fig:20step}.  The family of hyperplanes orthogonal to $\alpha_1^\vee$  (resp.~$\alpha_2^\vee$) are labeled in blue (resp.~red) in Figure \ref{fig:20step}.  We have chosen to label all hyperplanes in Figure \ref{fig:20step} with positive roots in $R^+_{af}$.

\begin{figure}
	\centering
		\begin{overpic}[scale = .35]{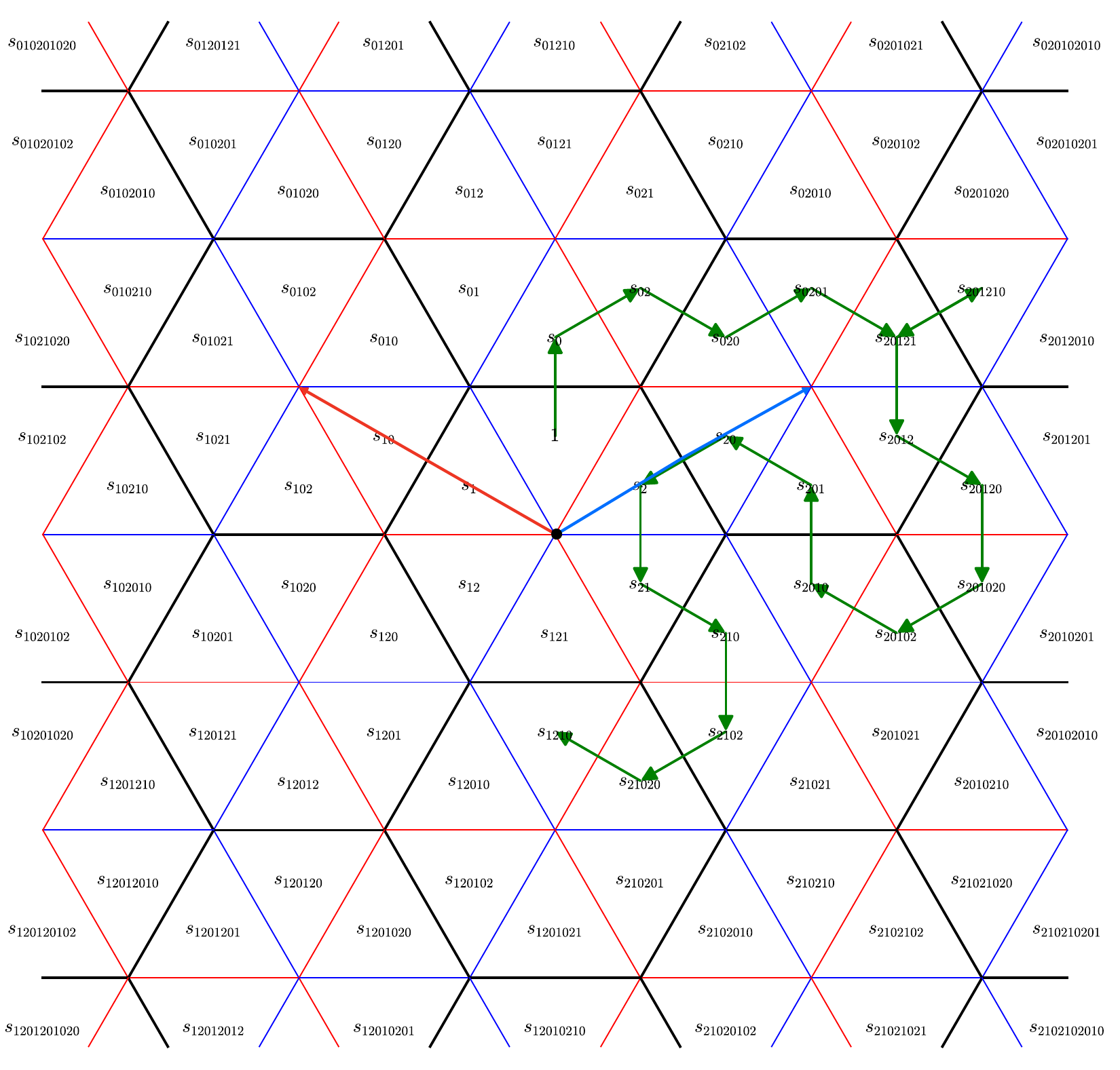}
	\put(91,88){\parbox{10cm}{$H_{-\alpha_1 - \alpha_2 + 2\delta}$}}
	\put(91,74.5){\parbox{10cm}{$H_{-\alpha_1 - \alpha_2 + 2\delta}$}}
	\put(91,61.5){\parbox{10cm}{$H_{-\alpha_1 - \alpha_2 + \delta}$}}
	\put(91,48.5){\parbox{10cm}{$H_{\alpha_1 + \alpha_2 }$}}
	\put(91,35.5){\parbox{10cm}{$H_{\alpha_1 + \alpha_2 + \delta}$}}
	\put(91,22.5){\parbox{10cm}{$H_{\alpha_1 + \alpha_2 + 2\delta}$}}
	\put(91,9){\parbox{10cm}{$H_{\alpha_1 + \alpha_2 + 3\delta}$}}

        \put(14,1){\parbox{10cm}{$\blue{H_{\alpha_1 +4\delta} }$}}
	\put(29,1){\parbox{10cm}{$\blue{H_{\alpha_1 +3\delta} }$}}
	\put(44.5,1){\parbox{10cm}{$\blue{H_{\alpha_1 +2\delta} }$}}	
	\put(60,1){\parbox{10cm}{$\blue{H_{\alpha_1 +\delta} }$}}
	\put(75,1){\parbox{10cm}{$\blue{H_{\alpha_1} }$}}
	\put(90,1){\parbox{10cm}{$\blue{H_{-\alpha_1 +\delta} }$}}
	
	  \put(2.5,88.5){\parbox{10cm}{$\red{H_{-\alpha_2 +4\delta} }$}}
	   \put(17.5,88.5){\parbox{10cm}{$\red{H_{-\alpha_2 +3\delta} }$}}
	   \put(32.5,88.5){\parbox{10cm}{$\red{H_{-\alpha_2 +2\delta} }$}}
	    \put(49,88.5){\parbox{10cm}{$\red{H_{-\alpha_2 +\delta} }$}}
	     \put(67.25,88.5){\parbox{10cm}{$\red{H_{\alpha_2} }$}}
	      \put(81,88.5){\parbox{10cm}{$\red{H_{\alpha_2 + \delta} }$}}
	      
	   \put(35,58){\parbox{10cm}{$\red{\alpha_2^\vee} $}}
	    \put(60.5,58){\parbox{10cm}{$\blue{\alpha_1^\vee} $}}
	   	\end{overpic}
	\caption{A non-minimal alcove walk ending in $\w$ where $w=s_{1210}$, having type $s_{02012001020201010202}$.}
	\label{fig:20step}
\end{figure}

 The reflection across $H_{\alpha+k\delta}$ defined by 
$ s_{\alpha, k}(v) = v - (\langle v, \alpha\rangle - k)\alpha^\vee$
 is an affine transformation of $V^*$, and 
the reflections $s_{\alpha, k}$ generate the affine Weyl group $W_{af}$.  The action of $W_{af}$ on $V^*$ preserves the collection of hyperplanes $H_{\alpha+k\delta}$. 
The reflection $s_{\alpha, k}$ associated to the affine root $\beta = \alpha+k\delta \in R_{af}$ can equivalently be expressed as the composition $s_\beta = s_\alpha t_{k\alpha^\vee}$, where $t_{k\alpha^\vee}$ denotes the translation in $V^*$ by the vector $k\alpha^\vee$, and $s_{\alpha}$ is the reflection across the hyperplane $H_{\alpha}$. Given any $\mu+m\delta \in R_{af}$, we have $t_{k\alpha^\vee}(\mu + m\delta) = (\mu+m\delta) - \langle k\alpha^\vee, \mu \rangle \delta,$ and therefore 
\begin{equation}\label{eq:affonaff}
s_{\alpha+k\delta}(\mu+m\delta) = \mu - \langle \alpha^\vee, \mu \rangle \alpha + \left( m- \langle k\alpha^\vee, \mu \rangle\right) \delta.
\end{equation}

 The complement in $V^*$ of the connected components of the hyperplanes $H_{\alpha+k\delta}$ are \emph{open alcoves}, and their closures in $V^*$ are simply called \emph{alcoves}. The affine Weyl group $W_{af}$ acts simply transitively on the set of alcoves in $V^*$. The base alcove $\base_0$ is defined as
$\base_0 = \{ v \in V^* \mid 0 \leq \langle v, \alpha \rangle \leq 1\ \text{for all}\ \alpha \in R^+ \};$ this is the alcove labeled by 1 in Figure \ref{fig:20step}.
Under the action of $W_{af}$ on $V^*$, the element $w \in W_{af}$ corresponds bijectively to the alcove $\w = w\base_0$, and we henceforth identify the elements of $W_{af}$ with their alcoves.  The codimension one faces of an alcove are called \emph{panels}.

 An \emph{alcove walk} is a sequence of alcoves $\gamma = (\base_0, \base_1,  \dots, \base_m)$  such that each successive pair of alcoves shares a panel.  If the final alcove $\base_m = w\base_0$ for $w \in W$, then we say that $\gamma$ is an alcove walk \emph{from} $\base_0$ \emph{to} $\w$ of \emph{length} $m$. See Figure \ref{fig:20step} for an example of an alcove walk from $\base_0$ to $\w$ for $w=s_{1210}$ of length 20, illustrated as a sequence of green arrows. Until we consider folded alcove walks in Section \ref{sec:folded}, we assume that  $\base_{j-1} \neq \base_j$ for all $j \in \{1, \dots, m\}$. There thus exists a unique $i \in I_{af}$ and $v \in W_{af}$ such that $\base_{j-1} = v\base_0$ and $\base_{j} = vs_i\base_0$.  The common panel between $\base_{j-1}$ and $\base_{j}$ is then said to have \emph{type} $i \in I_{af}$.  
 The \emph{type of the alcove walk} $\gamma$ is the word in $W_{af}$ defined by $\operatorname{type}(\gamma) = s_{i_1}s_{i_2} \cdots s_{i_m}$, where for each $1 \leq j \leq m$, the panel common to $\base_{j-1}$ and $\base_{j}$ has type $i_j \in I_{af}$.  The type of the alcove walk $\gamma$ corresponds to an expression for the element $w = s_{i_1} \cdots s_{i_m}$ such that $\base_m = w\base_0$.  If the length of the alcove walk equals $\ell(w)$, then the corresponding expression for $w$ is reduced, and we say that the alcove walk has \emph{minimal length}, equivalently that $\gamma$ is minimal.  The alcove walk in Figure \ref{fig:20step} is not minimal, and its type is recorded in the caption.
 
 Conversely, given any expression for $w = s_{i_1} \cdots s_{i_m} \in W_{af}$, we can define an associated alcove walk $\gamma_w = (\base_0, \base_1, \dots, \base_m)$
from the base alcove $\base_0$ to $\w = \base_m$ by setting $\base_j = s_{i_{1}}\cdots s_{i_{j}}\base_0$ for all $1 \leq j \leq m$.  In other words, $\gamma_w$ is the unique alcove walk such that $\type(\gamma_w) = s_{i_1}\cdots s_{i_m}$.  We also refer to an \emph{alcove walk $\gamma$ of type} $\w$, leaving implicit the choice of expression for $w$ determined by $\type(\gamma)$, since all relevant constructions will be shown to be independent of this choice.
 
 The \emph{inversion set of $w \in W_{af}$} is defined as
$\inv(w) = \left\{ \beta \in R_{af}^+ \ \middle|\  w^{-1}\beta \notin R_{af}^+ \right\}.$
The inversions of $w$ can also be written explicitly in terms of any chosen reduced expression $w = s_{i_1}\cdots s_{i_\ell}$; namely, $\inv(w) = \{ \beta _j \}$ where
\begin{equation}\label{E:betaj}
\beta_1 = \alpha_{i_1}, \quad \beta_2 = s_{i_1}\alpha_{i_2}, \quad \dots, \quad \beta_\ell = s_{i_1} \cdots s_{i_{\ell -1}} \alpha_{i_\ell}.
\end{equation}
In other words, the elements of $\inv(w)$ coincide with the roots $\beta \in R^+_{af}$ corresponding to the hyperplanes crossed at each step of the alcove walk $\gamma_w$ of type $w = s_{i_1}\cdots s_{i_\ell}$.
 
 We represent alcove walks as a sequence of concatenated arrows from $\base_{j-1}$ to $\base_{j}$, orthogonal to the hyperplane, say $H_{\beta_j+k_j\delta}$, containing the panel shared by $\base_{j-1} = v\base_0$ and $\base_{j} = vs_i\base_0$.  We say that the alcove walk \emph{crosses} hyperplane $H_{\beta_j+k_j\delta}$ at \emph{step} $j$, and we denote this crossing by a single arrow $\base_{j-1} \to \base_j$.   If $\ell(v) < \ell(vs_i)$, we index the crossing from $\base_{j-1}$ to $\base_{j}$ by the corresponding positive affine root $\left(\beta_{j}+k_{j}\delta\right)^+ \in R^+_{af}$.  Conversely, if $\ell(vs_i)<\ell(v)$, we index the crossing at step $j$ by the corresponding negative affine root $\left(\beta_{j}+k_{j}\delta\right)^- \in R^-_{af}$.  
 If $\ell(v)<\ell(vs_i)$, then the hyperplane crossing at step $j$ moves ``away'' from the base alcove $\base_0$, and we refer to this crossing as a \emph{forward step}.  Conversely, if $\ell(vs_i) < \ell(v)$, then step $j$ moves ``toward'' $\base_0$, and we refer to this crossing as a \emph{backward step}. For example, the fifth step of the alcove walk in Figure \ref{fig:20step} is a forward step indexed by $\alpha_2+ \delta \in R^+_{af}$, whereas the eighth step is backward and thus indexed by $\alpha_1+\alpha_2-\delta \in R^-_{af}$.

 The language of masks now extends to alcove walks as follows.  Let $\gamma = (\base_0, \base_1, \dots, \base_m)$ be an alcove walk of $\type(\gamma) = s_{i_1}\cdots s_{i_m} = w \in W_{af}$.  A mask $\varepsilon \in \{0,1\}^m$ picks out certain crossings in $\gamma$; namely, if $\varepsilon_j = 1$ the mask shows the crossing $\base_{j-1} \to \base_j$ at step $j$, and if $\varepsilon = 0$ the mask hides the crossing at step $j$. Note that the length of the walk $\gamma$ coincides with the length of the mask $\varepsilon$, and so we omit reference to the ambient space $\varepsilon \in \{0,1\}^m$ when the length of $\gamma$ is clear.  For any subinterval $J_k = \{j, \dots, j+k-1\}$ of indices in $\{1, \dots, m\}$, the submask $\varepsilon_{J_k} = (\varepsilon_j, \dots, \varepsilon_{j+k-1})$ corresponds to the subsequence $\sigma = (\base_j, \dots, \base_{j+k-1})$ of $\gamma$, which can be viewed as an alcove walk instead starting in alcove $\base_j$ and ending in $\base_{j+k-1}$. Given a subsequence $\sigma$ of $\gamma$, denote by $\varepsilon_\sigma$ the subsequence of $\varepsilon$ which shows or hides only those crossings in $\sigma$. The type of the subsequence $\sigma$  is given by the subexpression $w^{\varepsilon_\sigma} = s_{i_{j}}^{\varepsilon_j}\cdots s_{i_{j+k-1}}^{\varepsilon_{j+k-1}}$ corresponding to panels crossed by $\sigma$.

\section{Alcove Walks and Affine Root Algebra}\label{sec:polys}

This section introduces a new polynomial algebra on alcove walks, by assigning a particular product of affine roots to each mask for a fixed alcove walk to $\w$.  When summing over all masks of minimum support for a given subword of $w$, the resulting polynomial is shown in Proposition \ref{prop:Indep} to be independent of the initial choice of walk, including that this walk need not be minimal.  This analysis gives rise to a family of polynomial functions $\psi^v$ for any $v \in W_{af}$, whose value $\psi^v(w)$ fixes a walk to $\w$ and records the hyperplane crossings which are shown by any mask for $v$; see Definition \ref{def:psidef} for the precise formula.

\subsection{Products of affine roots from alcove walks}

In this section, we define an element of $\ZZ[\alpha_0,\alpha_1, \dots, \alpha_{n-1}]$ corresponding to each mask for a given alcove walk.  We record these polynomials as products of affine roots of the form $\alpha+k\delta$ via the relation $\delta = \alpha_0+\theta$.

   \begin{definition}\label{def:Psi-gamma}
  Let  $\gamma = (\base_0, \base_1, \dots, \base_m)$ be an alcove walk of
$\operatorname{type}(\gamma) = s_{i_1}s_{i_2} \cdots s_{i_m}$.
 To each mask $\varepsilon = (\varepsilon_1, \dots, \varepsilon_m) \in \{0,1\}^m$,
  we associate the product of affine roots
\begin{equation}\label{eq:Psi-gamma}
\Psi_\gamma^{\varepsilon} = \prod\limits_{ \varepsilon_j=1 }\left(\beta_{i_j}+k_{i_j}\delta\right)^{\pm}.
\end{equation}
The multiplicands in \eqref{eq:Psi-gamma} are those roots indexing the crossings of $\gamma$ at steps corresponding to the nonzero entries in $\varepsilon$. The $\pm$ sign serves as a reminder that if the step is forward (resp.~backward), then we choose the corresponding positive (resp.~negative) affine root indexing the hyperplane crossed at each step. If a mask $\varepsilon$ hides all crossings, then we set $\Psi^\varepsilon_\gamma = 1$. 
\end{definition}

\begin{figure}
	\centering
	\begin{overpic}[scale = 0.275]{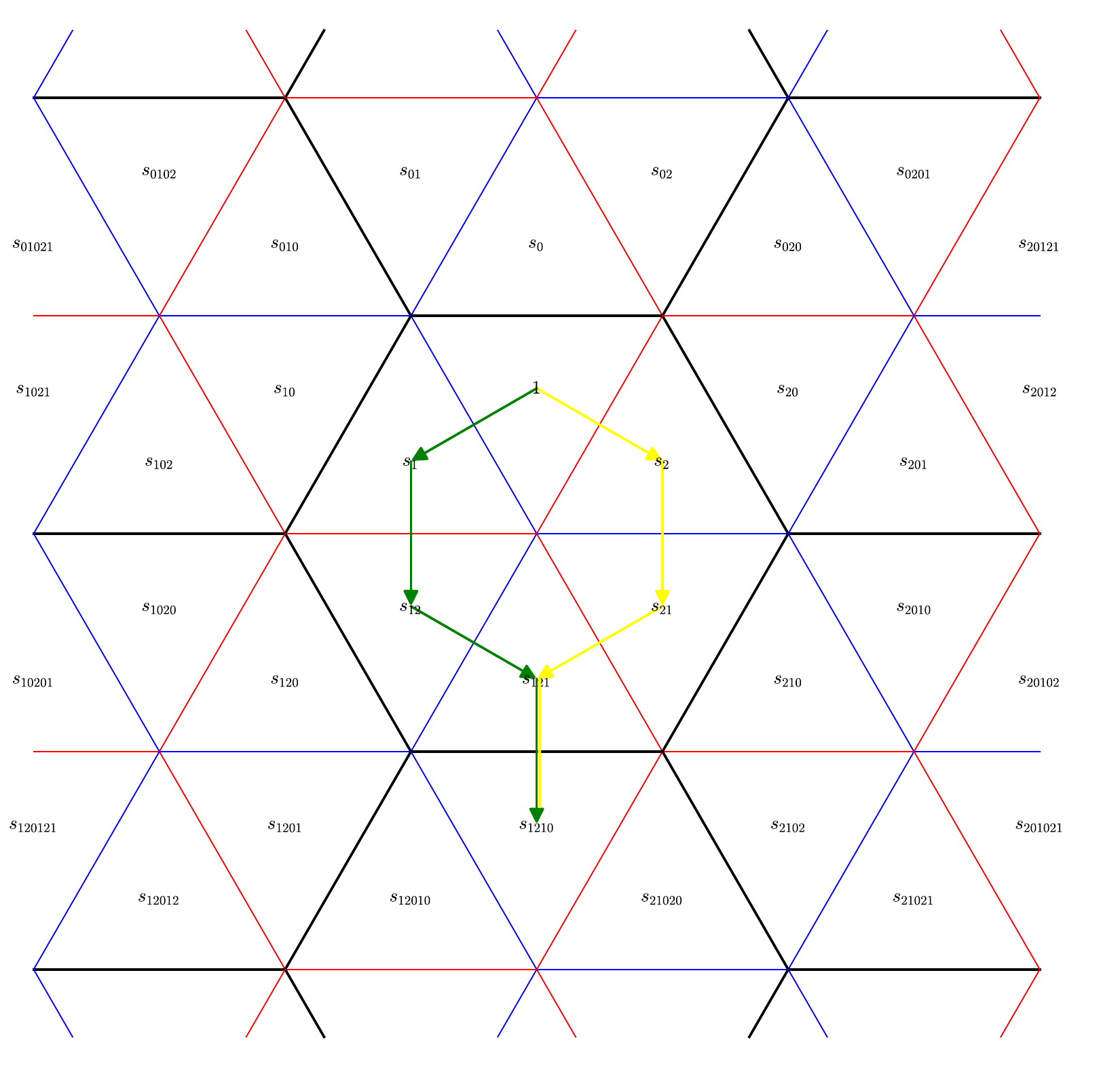}
	\put(93,46){\parbox{10cm}{$H_{\alpha_1 + \alpha_2 }$}}
	\put(93,27){\parbox{10cm}{$H_{\alpha_1 + \alpha_2 + \delta}$}}
	\put(73,1){\parbox{10cm}{$\blue{H_{\alpha_1 }}$}}
	\put(73,91){\parbox{10cm}{$\red{H_{\alpha_2 }}$}}
	\put(61,50){\parbox{10cm}{$\gamma'$}}
	\put(33,50){\parbox{10cm}{$\gamma$}}
	\end{overpic}
	\caption{Both alcove walks end in the same alcove and have the same length, but their types differ by one  braid relation. The walk $\gamma$ in green has type $s_{1210}$, and the walk $\gamma'$ in yellow has type $s_{2120}$.}
	\label{fig:braid}
\end{figure}

\begin{example0}\label{ex:psicalc}
Here we present two examples illustrating Definition \ref{def:Psi-gamma}. Let $\gamma$ be the alcove walk of type $w=s_1s_2s_1s_0$ depicted in green in Figure \ref{fig:braid}, and let $v = s_1s_0 < w$. Consider the mask $\varepsilon = (1,0,0,1)$ on $\gamma$, which satisfies $w^\varepsilon = s^1_1s^0_2s^0_1s^1_0 = v$.  
The first crossing shown by $\varepsilon$ occurs at step 1 where $\gamma$ crosses the hyperplane $H_{\alpha_1}$.  Since this crossing is a forward step, we record $\alpha_1 \in R^+_{af}$ as the first entry of the product \eqref{eq:Psi-gamma}.  The second crossing shown by $\varepsilon$ occurs at the forward step 4, which is indexed by $\alpha_1+\alpha_2+\delta \in R^+_{af}$.
Therefore Equation \eqref{eq:Psi-gamma} gives us \[\Psi^{\varepsilon}_{\gamma} = \alpha_1 (\alpha_1 + \alpha_2 + \delta).\] 
Similarly, the mask $\varepsilon' = (0,0,1,1)$ satisfies $w^{\varepsilon'} = s_1^0s_2^0s_1^1s_0^1 = v$, and we compute 
\[\Psi^{\varepsilon'}_{\gamma} = \alpha_2 (\alpha_1 + \alpha_2 + \delta).\] 
In fact, these are the only two masks for $v$ on the chosen walk $\gamma$ of type $\w$.  Further note that  $|\varepsilon| = |\varepsilon'| = 2 = \ell(v)$ in both cases.
\end{example0}

\subsection{Polynomial algebra from alcove walks}

Our first result shows that, when summing the expressions $\Psi_\gamma^\varepsilon$ over all hyperplane crossings whose type is a given subword, the summation is independent of the initial choice of the alcove walk to $\w$.  In particular, the alcove walk to $\w$ need not be minimal.

\begin{proposition}\label{prop:Indep}
Let $w \in W_{af}$, and let $\gamma = (\base_0, \base_1, \dots, \base_m)$ and $\gamma' = (\base_0, \base'_1, \dots, \base'_{m'})$ be any two alcove walks from $\base_0$ to $\w$. Then for any $v \in W_{af}$, we have
\begin{equation}\label{eq:Indep}
\sum\limits_{\substack{|\varepsilon| = \ell(v) \\ w^\varepsilon = v}} \Psi_\gamma^\varepsilon = \sum\limits_{\substack{ |\varepsilon'| = \ell(v)  \\ w^{\varepsilon'} = v}} \Psi_{\gamma'}^{\varepsilon'}.
\end{equation}
\end{proposition}

Before proceeding to the proof, we illustrate the proposition in a small example, where $\gamma$ and $\gamma'$ are related by a single braid relation.

 \begin{example0}\label{ex:braid}
To understand how this calculation is affected by a braid move, consider the following example in affine type $A_2$. We return to the green alcove walk $\gamma$ of type $w=s_{1210}$ from Example \ref{ex:psicalc} shown in Figure \ref{fig:braid}.  We also consider the yellow alcove walk $\gamma'$ of type $s_{2120}$ ending in the same alcove $\w$. Finally let $v = s_{10} < w$. 

As computed in Example \ref{ex:psicalc}, there are exactly two masks for $v$ on $\gamma$ of type $\w$ such that $|\varepsilon| = |\varepsilon'| = 2 = \ell(v)$. Therefore, by our calculations in Example \ref{ex:psicalc}, we have
\[ 
\sum\limits_{\substack{|\varepsilon| = \ell(v) \\ w^\varepsilon = v}} \Psi_\gamma^\varepsilon  = \alpha_1 (\alpha_1 + \alpha_2 + \delta) + \alpha_2 (\alpha_1 + \alpha_2 + \delta) = (\alpha_1+\alpha_2)(\alpha_1+\alpha_2+\delta).
\]
To illustrate Proposition \ref{prop:Indep}, we now consider the walk $\gamma'$ whose type differs from $\type(\gamma)$ by a single braid relation.  Here instead only one mask $\varepsilon'' = (0,1,0,1)$ satisfies $w^{\epsilon''} = v$ and $|\varepsilon''| =\ell(v)$.  The first crossing shown by $\varepsilon''$ occurs at step 2 where $\gamma'$ crosses the hyperplane $H_{\alpha_1+\alpha_2}$, and the second crossing shown by $\varepsilon''$ is indexed by $(\alpha_1+\alpha_2+\delta)$ which is common to $\gamma$.  Therefore, 
\[
\sum\limits_{\substack{ |\varepsilon''| = \ell(v)  \\ w^{\varepsilon''} = v}} \Psi_{\gamma'}^{\varepsilon''} =  (\alpha_1 + \alpha_2)(\alpha_1+\alpha_2+\delta)  = \sum\limits_{\substack{|\varepsilon| = \ell(v) \\ w^\varepsilon = v}} \Psi_\gamma^\varepsilon,
\]
as the proposition asserts.
\end{example0}

With this example in hand, we now proceed to the proof in the general case.

\begin{proof}[of Proposition \ref{prop:Indep}]
By a result of Tits also known as Matsumoto's theorem \cite{Matsumoto, Tits}, any given expression for an element $w \in W_{af}$ can be obtained from any other expression for $w$ by performing a finite sequence of braid relations. Letting $s$ and $t$ denote simple generators, this says any two given expressions differ by a sequence of quadratic relations of the form $ss = 1$, or braid substitutions $(sts\cdots) = (tst \cdots)$ where there are $m_{st}$ simple generators in each expression.  

Therefore, it suffices to assume that $\type(\gamma)$ and $\type(\gamma')$ differ by applying exactly one braid relation.  Denote by $\sigma$ the subsequence of $\gamma$ which differs in type from $\gamma'$, and denote by $\sigma'$ the corresponding subsequence of $\gamma'$.  There are two cases to consider.  First suppose that $\type(\sigma) = ss$ while $\type(\sigma') = 1$, so that $m'=m-2$.  That is, $\gamma$ is the same as $\gamma'$ except that the subsequence $\sigma$ is replaced by $\sigma' = \emptyset$. 

Now consider any mask $\varepsilon \in \{0,1\}^m$ such that $w^\varepsilon = v$ and $|\varepsilon| = \ell(v)$.  There are again two cases to consider.  First suppose that the submask $\varepsilon_\sigma$ hides both crossings in $\sigma$; equivalently, $\varepsilon_\sigma = 00$.  Every such $\varepsilon$ corresponds to a unique mask $\varepsilon' \in \{0,1\}^{m'}$, obtained by simply deleting the subsequence $\varepsilon_\sigma = 00$ from $\varepsilon$. Moreover, all masks $\varepsilon'$ for $\gamma'$ such that $w^{\varepsilon'} = v$ and $|\varepsilon'| = \ell(v)$ arise in this way.  Since the $\sigma$-crossings do not contribute to $\Psi_\gamma^\varepsilon$ in this case, and since all hyperplane crossings outside of $\sigma$ are identical in both $\gamma$ and $\gamma'$, then we have $\Psi^{\varepsilon}_\gamma = \Psi^{\varepsilon'}_{\gamma'}$ and thus
\begin{equation*}
\sum\limits_{\substack{\varepsilon_\sigma =00 \\ |\varepsilon| = \ell(v) \\ w^\varepsilon = v}} \Psi_\gamma^\varepsilon = \sum\limits_{\substack{ |\varepsilon'| = \ell(v)  \\ w^{\varepsilon'} = v}} \Psi_{\gamma'}^{\varepsilon'},
\end{equation*}

To complete the proof in the case where $\type(\sigma) = ss$, it remains to show that
\begin{equation}\label{eq:zero}
\sum\limits_{\substack{ \varepsilon_\sigma \neq 00 \\ |\varepsilon| = \ell(v) \\  w^\varepsilon = v}} \Psi_\gamma^\varepsilon = 0.
\end{equation}
There are two further cases to consider.  Suppose that the submask $\varepsilon_\sigma$ shows both crossings in $\sigma$, meaning that $\varepsilon_\sigma = 11$. Then $\gamma$ crosses the same hyperplane twice at the two steps corresponding to $\sigma$.  However, the corresponding subword $w^\varepsilon$ is of the form $w^\varepsilon = s_{j_1} \cdots ss \cdots s_{j_l} = v$ and cannot be a reduced expression.  In particular, $|\varepsilon| > \ell(v)$, and so masks such that $\varepsilon_\sigma = 11$ do not occur in the sum in \eqref{eq:zero}.

Finally, consider the case where $\varepsilon_\sigma = 10$. Each such $\varepsilon$ corresponds to a unique mask $\overline{\varepsilon}$ for $\gamma$, obtained by switching the subsequence $\varepsilon_\sigma = 10$ to $\overline{\varepsilon}_\sigma = 01$ and leaving all other entries the same.  Moreover, all masks  $\overline{\varepsilon}$ for $\gamma$ such that $\overline{\varepsilon}_\sigma =  01$ with $w^{\overline{\varepsilon}} = v$ and $|\overline{\varepsilon}| = \ell(v)$ arise in this way.  If the first crossing of $\sigma$ occurs at step $j$ in $\gamma$, we can factor $\Psi^\varepsilon_\gamma = (\beta_{i_j}+k_{i_j}\delta)\prod\limits_{1=\varepsilon_{\ell\neq j}}\left( \beta_{i_\ell}+k_{i_\ell}\delta\right)$.  On the other hand, by definition $\Psi^{\overline{\varepsilon}}_\gamma = (-\beta_{i_j}-k_{i_j}\delta)\prod\limits_{1=\varepsilon_{\ell\neq j}}\left( \beta_{i_\ell}+k_{i_\ell}\delta\right)$, using the facts that $\type(\sigma) = ss$ and that $\varepsilon$ and $\overline{\varepsilon}$ agree on all crossings outside of $\sigma$.  Therefore, $\Psi^\varepsilon_\gamma + \Psi^{\overline{\varepsilon}}_\gamma = 0$, completing the proof of \eqref{eq:zero} and thus concluding  the case when $\type(\sigma) = ss$.

Now suppose that $\gamma$ is the same as $\gamma'$ except that the subsequence $\sigma$ is of $\type(\sigma) = (sts\cdots)$ and is replaced by $\sigma'$ of $\type(\sigma') = (tst\cdots)$, where both expressions use $m_{st}$ generators. We treat the case where $m_{st}$ is even first.  

Consider any mask $\varepsilon$ for $\gamma$ such that $w^\varepsilon = v$ and $|\varepsilon| = \ell(v)$. When $m_{st}$ is even, each such $\varepsilon$ corresponds to a unique mask $\varepsilon'$ for $\gamma'$ obtained by keeping all entries outside $\sigma$ the same, and replacing the submask $\varepsilon_\sigma = b_1\cdots b_{m_{st}}$ by the same binary sequence read backwards; that is, define $\varepsilon'_{\sigma'} = b_{m_{st}}\cdots b_1$.  Moreover, for $m_{st}$ even, all masks $\varepsilon'$ for $\gamma'$ such that $w^{\varepsilon'} = v$ and $|\varepsilon'| = \ell(v)$ arise in this way.

 The expression $\Psi^\varepsilon_\gamma$ factors into a product of those affine roots indexing steps within $\sigma$ and those for affine roots in the complement $\gamma - \sigma$; we write $\Psi^\varepsilon_\gamma = \Psi^\varepsilon_\sigma \Psi^\varepsilon_{\gamma - \sigma}$.  We claim that $\Psi^\varepsilon_\gamma = \Psi^{\varepsilon'}_{\gamma'}$, which suffices to prove \eqref{eq:Indep} in case $m_{st}$ is even, since we have shown the masks indexing the summations to be in bijection. By construction, $\Psi_{\gamma - \sigma}^\varepsilon = \Psi_{\gamma'-\sigma'}^{\varepsilon'}$.  Therefore, it remains to prove that $\Psi^\varepsilon_\sigma = \Psi^{\varepsilon'}_{\sigma'}$ for even $m_{st}$.

 After relabeling, we may assume without loss of generality that both $\sigma$ and $\sigma'$ start in the base alcove $\base_0$.  The types of both $\sigma$ and $\sigma'$ coincide with two different reduced expressions for the longest word $w^0_{st}$ in the standard parabolic subgroup $W_{\{s,t\}}$, which is the dihedral group generated by $s$ and $t$.  The alcove walks $\sigma$ and $\sigma'$ thus cross the same hyperplanes, just in the opposite order.  Therefore, $\varepsilon_\sigma$ and $\varepsilon'_{\sigma'}$ also show (and hide) the exact same hyperplanes, recorded in the opposite order.  Therefore, $\Psi^\varepsilon_\sigma = \Psi^{\varepsilon'}_{\sigma'}$ for even $m_{st}$.

 Finally, consider the remaining case where $\type(\sigma) = sts$ and $\type(\sigma') = tst$ as in Example \ref{ex:braid}, recalling that in the affine context, this is the only situation in which $m_{st}$ is odd. Consider any mask $\varepsilon$ for $\gamma$ such that $w^\varepsilon = v$ and $|\varepsilon| = \ell(v)$. There are eight possible binary subsequences $\varepsilon_\sigma \in \{0,1\}^3$.  Note that if $\varepsilon_\sigma = 101$, however, then $\type(\sigma) = ss$ and $|\varepsilon| > \ell(v)$, and so $\varepsilon$ does not index a summand.  If $\varepsilon_\sigma \in \{ 000,110,011,111 \}$, then the same argument holds as for $m_{st}$ even; namely, define $\varepsilon'_{\sigma'}$ by reading the sequence $\varepsilon_\sigma$ backwards, whence $\Psi^\varepsilon_\sigma = \Psi^{\varepsilon'}_{\sigma'}$.
 
 Next consider $\varepsilon_\sigma = 010$.  We shall pair $\varepsilon_\sigma$ with two different masks for $\sigma'$; namely $\varepsilon'_1 = 100$ and $\varepsilon'_2 = 001$.  Compute directly in this case that $\Psi^\varepsilon_\sigma = \alpha_s+\alpha_t$,  while $\Psi^{\varepsilon'_1}_{\sigma'} = \alpha_t$ and $\Psi^{\varepsilon'_2}_{\sigma'} = \alpha_s$, so that $\Psi^\varepsilon_\sigma = \Psi^{\varepsilon'_1}_{\sigma'}+\Psi^{\varepsilon'_2}_{\sigma'}$ as in Example \ref{ex:braid}.  By symmetry, the remaining two masks $\varepsilon_1 = 100$ and $\varepsilon_2 = 001$ pair with $\varepsilon' = 010$ to give $\Psi^{\varepsilon_1}_{\sigma}+\Psi^{\varepsilon_2}_{\sigma} = \Psi^{\varepsilon'}_{\sigma'}$. Having covered all possible cases for the submasks $\varepsilon_\sigma$ and $\varepsilon'_{\sigma'}$, by factoring out the common product $\Psi_{\gamma - \sigma}^\varepsilon = \Psi_{\gamma'-\sigma'}^{\varepsilon'}$ we obtain \eqref{eq:Indep} in this final case, concluding the proof of the proposition.
 \end{proof}

\begin{remark}
Proposition \ref{prop:Indep} both refines and generalizes the analogous Lemma 4.1 of \cite{Billey}, by considering non-reduced expressions for $w$ and making more precise the relationship between corresponding summands affected by any given braid move. 
\end{remark}

In light of Proposition \ref{prop:Indep}, the summation in \eqref{eq:Indep} only depends on the choice of the elements $w,v \in W_{af}$ themselves, and not on any particular expressions for them, as the following definition formalizes.  
\begin{definition}\label{def:psidef}
Let $w, v \in W_{af}$.  Fix $\gamma = (\base_0, \base_1, \dots, \base_m)$ any alcove walk to $\w$.  Define
\begin{equation}\label{eq:loc}
\psi^v(w) = \sum\limits_{\substack{ |\varepsilon| = \ell(v) \\ w^\varepsilon = v}} \Psi^\varepsilon_\gamma.
\end{equation}
If $w = e$, then we may choose the trivial walk $\gamma = (\base_0)$ and the empty mask $\varepsilon_\emptyset$, in which case $\Psi^{\varepsilon_\emptyset}_\gamma =1$. If the sum is empty, then $\psi^v(w) = 0$.  
\end{definition}
Note that by definition, $\psi^v(w) \in \ZZ[\alpha_0, \alpha_1, \dots, \alpha_{n-1}]$ using the relation $\delta = \alpha_0+\theta$.  In the next section, we will recognize $\psi^v$ as the functions of Kostant and Kumar corresponding to localizations of affine Schubert classes at torus-fixed points \cite{KostantKumar}.

\section{Alcove Walks and Localization}\label{sec:GKMresults}

We begin this section by briefly reviewing the setup of GKM theory for the affine flag variety, referring the reader to expository references such as \cite{Kumar,kSchur,TymGKM} for more details.  We then prove in Theorem \ref{thm:GKM} that the functions $\psi^v$ defined in terms of (not necessarily minimal) alcove walks satisfy the GKM condition, placing them in the image of the localization map.  In Theorem \ref{thm:Local}, we further identify $\psi^v$ as the image of the affine Schubert class $[X_v]$ under this localization.

\subsection{Affine flag varieties}\label{sec:GKMintro}

We begin by connecting the geometry and combinatorics of affine Weyl groups and their root systems to some additional group theory and topology.

Let $G$ be a semisimple algebraic group over $\CC$. We fix a Borel subgroup $B$ such that the roots in $R^+$ are positive with respect to the opposite Borel subgroup $B^-$.  This choice ensures that the elements of $R$ which lie in the half-space bounded by $H_{\theta}$ containing the base alcove $\base_0$ are the positive finite roots $R^+$.  The Borel contains a split maximal torus $T$, and the finite Weyl group is then precisely $W = N_G(T)/T$.  For example, we could consider $G=SL_n$ with $B$ the subgroup of upper-triangular matrices and $T$ the subgroup of diagonal matrices, in which case $B^-$ is the subgroup of lower-triangular matrices and $W = S_n$ is the symmetric group.

In the affine context, we will work over the field of Laurent series $F = \CC((t))$, having ring of integers $\mathcal{O} = \CC[[t]]$. To distinguish the various affine versions of the Borel subgroup and maximal torus, we write $B_{af} = B(F)$ and $T_{af} = T(F)$, which are subgroups of $G(F)$.  The (extended) affine Weyl group is then $W_{af} = N_GT(F)/T(\mathcal{O})$.  If $G = SL_n$, then $B_{af}$ (resp.~$T_{af}$) are still upper-triangular (resp.~diagonal) matrices, though now with Laurent series entries, and $W_{af} = \widetilde{S}_n$ is the affine symmetric group.

Define the \emph{(positive) Iwahori subgroup} $I$ of $G(F)$ as the preimage of the opposite Borel subgroup $B^-(\CC)$ under the projection map $G(\mathcal{O}) \to G(\CC)$.  Equivalently, $I$ is the stabilizer of the base alcove $\base_0$ under the natural left action of $G(F)$ on $V^*$. If $G = SL_n$ and $B$ is the subgroup of upper-triangular matrices, then 
 \begin{equation*} I = \begin{pmatrix} \mathcal{O}^{\times} & t\mathcal{O} & \dots & t\mathcal{O} \\ \mathcal{O} & \mathcal{O}^{\times} & \dots & t\mathcal{O} \\ \vdots & \vdots & \ddots & \vdots \\ \mathcal{O} & \mathcal{O} & \dots & \mathcal{O}^{\times} \\ \end{pmatrix}.\end{equation*} 
 If we expand our power series in $t^{-1}$ instead, we obtain a corresponding collection of groups and subgroups.  Denoting $\mathcal{O}^- = \CC[[t^{-1}]]$, we define the \emph{negative Iwahori subgroup} $I^-$ of $G(F)$ to be the preimage of the Borel subgroup $B(\CC)$ under the projection $G(\mathcal{O}^-) \to G(\CC)$. There are two resulting affine versions of the Bruhat decomposition expressing $G(F)$ as a union of \emph{(opposite) affine Schubert cells}:
\begin{equation}\label{eq:bruhat}
G(F) = \bigsqcup\limits_{w \in W_{af}} IwI = \bigsqcup\limits_{v \in W_{af}} I^- v I.
\end{equation}

The quotient $G(F)/I$ is called the \emph{affine flag variety}, which we typically denote by $G/I$ for brevity. This homogeneous space is a projective ind-variety, which has a distinguished collection $X_v = \overline{I^-vI/I}$ of subvarieties called \emph{(opposite) affine Schubert varieties}, indexed by the elements $v \in W_{af}$ of the affine Weyl group; see \cite[Sec.~7.1]{Kumar}.
The affine flag variety admits a $T_{af}$-equivariant cohomology ring $H^*_{T_{af}}(G/I)$.  Kostant and Kumar provide an explicit algebraic construction of the ring $H^*_{T_{af}}(G/I)$ in the more general Kac-Moody setting, as a free module over the polynomial ring $S = H^*_{T_{af}}(\operatorname{pt})$, having an additive basis given by the affine Schubert classes $[X_v]$.  We refer the reader to \cite[Sec.~11.3]{Kumar} for more details.

\subsection{GKM theory in affine flag varieties}\label{sec:GKMintro}

Kostant and Kumar defined a family of functions which determine the ring structure for the equivariant cohomology of any Kac-Moody flag variety in \cite{KostantKumar}.  We now briefly review this correspondence, referred to as \emph{GKM theory} due to the foundational work of Goresky, Kottwitz, and MacPherson \cite{GKM}.  We refer the reader to \cite[Section 11.3]{Kumar} or \cite[Chapter 4]{kSchur} for more details in the Kac-Moody setting. For an introduction to GKM theory in a slightly less general context, we also recommend the survey by Tymoczko  \cite{TymGKM}.

The inclusion of the $T_{af}$-fixed points into the affine flag variety $G/I$, themselves also indexed by elements $w \in W_{af}$, induces an injection 
\begin{equation}\label{eq:AffLoc}
H^*_{T_{af}}(G/I) \hookrightarrow H^*_{T_{af}}(G/I)^{T_{af}} \cong \bigoplus_{w \in W_{af}} \ZZ[\alpha_0,\alpha_1, \dots, \alpha_{n-1}]. 
\end{equation}
By definition, the induced map sends the Schubert class $[X_v]$ to the (infinite) collection of localizations $[X_v] |_w$ at the $T_{af}$-fixed points.  As such, the Schubert class $[X_v] \in H^*_{T_{af}}(G/I)$ is identified with a function 
\[
\xi^v : W_{af} \to \ZZ[\alpha_0,\alpha_1, \dots, \alpha_{n-1}].
\]
  When evaluated at $w \in W_{af}$, the polynomial $\xi^v(w)$ is homogeneous of degree $\ell(v)$. 
Altogether, this \emph{localization map} from Equation \eqref{eq:AffLoc} identifies each $T_{af}$-equivariant Schubert class with the corresponding polynomial-valued function $[X_v] \mapsto \xi^v$.   

The ring structure of $H^*_{T_{af}}(G/I)$ with respect to the affine Schubert basis is completely determined by the pointwise product of the family of functions $\xi^v$. Moreover, the image of  $H^*_{T_{af}}(G/I)$ in this polynomial ring has an incredibly simple characterization in terms of the \emph{GKM condition} reviewed in Equation \eqref{eq:GKM} below.  This remarkable approach is due to Goresky, Kottwitz, and MacPherson for certain algebraic varieties \cite{GKM}, was developed in the Kac-Moody setting by Kostant and Kumar \cite{KostantKumar}, and further generalized by Henriques, Harada, and Holm \cite{HHH}; see also \cite[Theorem 9.2]{GKMAffSpringer} for the context of affine Springer fibers.  For the reader interested in the small-torus equivariant cohomology instead, setting the null root $\delta \mapsto 0$ recovers the ring structure in $H^*_{T}(G/I)$.

The polynomials $\xi^v(w)$ can be computed in several ways, including the original method from \cite[Theorem 4]{Billey} often referred to as \emph{Billey's formula}, generalized to Kac-Moody flag varieties by an argument of Kumar included in the appendix; see also Proposition 3.10 in \cite[Chapter 4]{kSchur}. We also recommend the survey on Billey's formula by Tymoczko \cite{TymBilley}.  Andersen, Jantzen, and Soergel \cite{AJS} provided an independent formulation in their work on representations of quantum groups. See also the work of Goldin-Tolman  \cite{goldin2009towards} and Guillemin-Zara \cite{guillemin2002combinatorial} providing formulas for localizations of canonical cohomology classes by summing over paths in the GKM graph. One primary goal of the present paper is to offer a new perspective on Billey's formula, which is critical to understanding the ring structure in $H^*_{T_{af}}(G/I)$.

\subsection{Alcove walks and the GKM condition}\label{sec:GKM}

Our first theorem proves that the functions $\psi^v$ defined in \eqref{def:psidef}  satisfy the GKM condition.  In particular, each $\psi^v$ lies in the image of the localization map \eqref{eq:AffLoc}, and thus corresponds to a cohomology class in the ring $H^*_{T_{af}}(G/I)$.

\begin{figure}
	\begin{overpic}[scale=0.38]{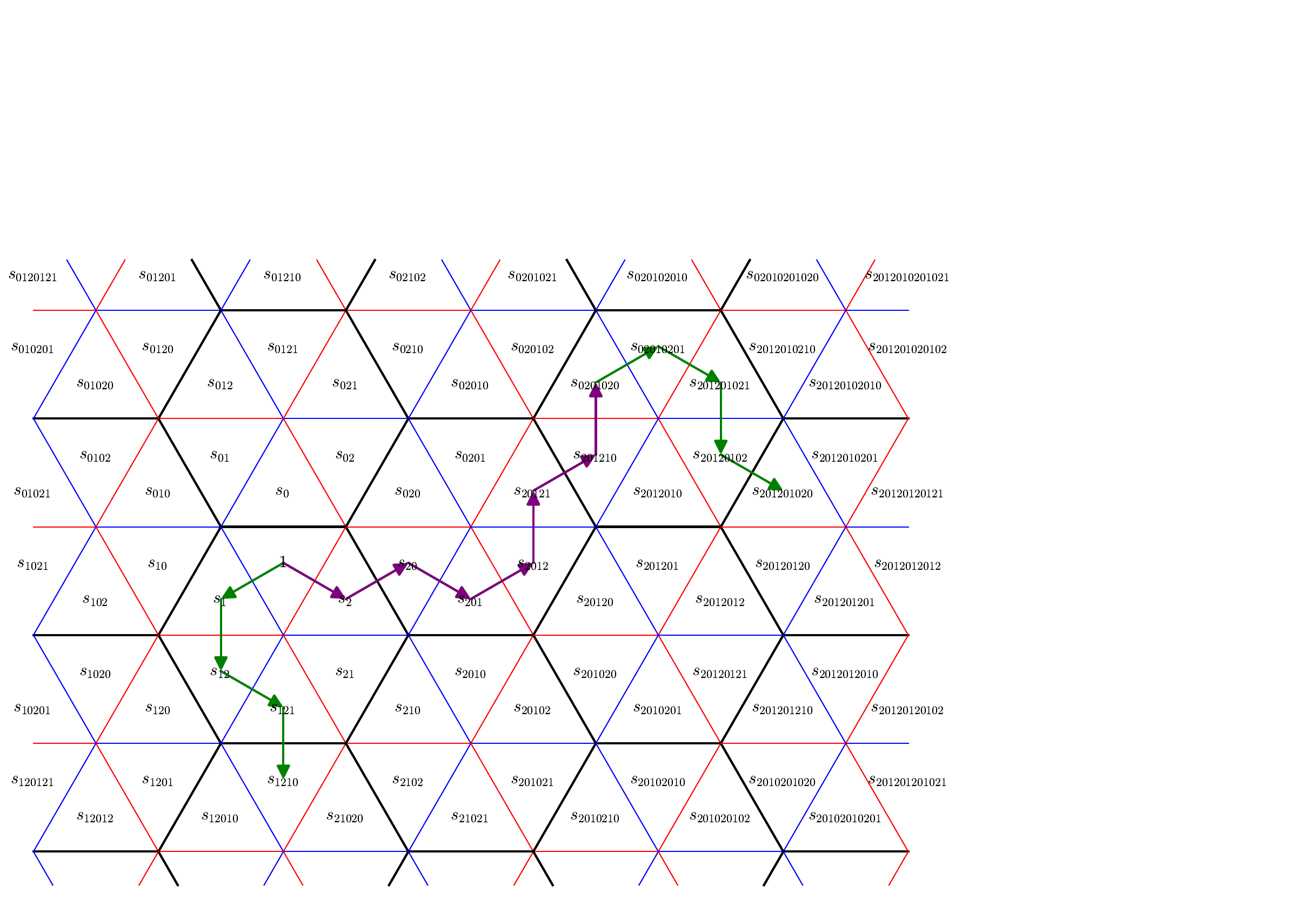}
	
	\put(68,1){\parbox{10cm}{$\blue{H_{-\alpha_1 +2\delta}}$}}
	\put(19,28){\parbox{10cm}{$\gamma_w$}}
	\put(48,34){\parbox{10cm}{$\gamma_{s_\beta}$}}
	\end{overpic}
\caption{The above image shows a walk to $\w$ where $w = s_{1210}$ in green, and a walk in purple corresponding to the palindromic expression $s_\beta = s_{2012102}$ for the reflection over the hyperplane $H_{\beta} = H_{-\alpha_1+2\delta}$. In our running example, we will use the (non-minimal) alcove walk of type $\textbf{s}_{\beta} \w$ coming from the concatenation of these two walks: first the purple, followed by the green walk.}
\label{fig:sbetaw}
\end{figure}

\begin{theorem}\label{thm:GKM}
Let $w,v \in W_{af}$. Then the functions $\psi^v$ satisfy the GKM condition; that is, for any $\beta \in R_{af}$,  we have
\begin{equation}\label{eq:GKM}
\beta \ \vert \ \left( \psi^v(w) - \psi^v(s_\beta w) \right).
\end{equation}
\end{theorem}

\begin{proof}
Without loss of generality, we assume $\beta \in R^+_{af}$.
Since $s_\beta$ is a reflection in a Coxeter group, there exist $u \in W_{af}$ and $s_i$ for some $i \in I_{af}$ such that $s_\beta = us_iu^{-1}$.
We can thus choose a palindromic reduced expression $s_\beta = s_{u_1}\cdots s_{u_r} s_i s_{u_r}\cdots s_{u_1}$, which determines an alcove walk $\gamma_{s_{\beta}}$ of this type; see the purple walk in Figure \ref{fig:sbetaw}. In addition, choose any expression for $w = s_{i_1}\cdots s_{i_m}$, and let $\gamma_w$ be the alcove walk of this type; see the green walk starting from $\base_0$ in Figure \ref{fig:sbetaw}.  Let $\gamma_{s_{\beta}w}$ be the alcove walk of type $\textbf{s}_\beta \w$ obtained by concatenating the purple walk $\gamma_{s_\beta}$ followed by the green walk $\gamma_w$; this concatenation is also shown in Figure \ref{fig:sbetaw}. Note that the two green walks in Figure \ref{fig:sbetaw} are related by the reflection $s_\beta$ across hyperplane $H_\beta$.

We now use Equation \eqref{eq:loc} to compute the polynomials $\psi^v(w)$ and $\psi^v(s_\beta w)$, using the walks $\gamma_{s_\beta}$ and $\gamma_{s_{\beta}w}$, respectively.  We first consider those terms $\Psi^\varepsilon_{\gamma_{s_\beta w}}$ of $\psi^v(s_\beta w)$ associated to masks which show crossings in $\gamma_{s_\beta}$.  Let $\varepsilon \in \{ 0,1 \}^{2r+m+1}$ be a mask with $(s_\beta w)^{\varepsilon} = v$ and $|\varepsilon| = \ell(v)$,  and such that $\varepsilon$ shows at least one crossing of $\gamma_{s_{\beta}}$. Since we obtained $\gamma_{s_\beta w}$ by concatenation, the expression $\Psi^\varepsilon_{\gamma_{s_\beta w}}$ factors into a product of those affine roots indexing steps within $\gamma_{s_\beta}$ and those from $\gamma_w$, meaning $\Psi^\varepsilon_{\gamma_{s_\beta w}}=\Psi^\varepsilon_{\gamma_{s_\beta}}\Psi^\varepsilon_{\gamma_{w}}$.  The submask $\varepsilon'$ of $\varepsilon$ restricted to $\gamma_{s_\beta}$ corresponds to a subword $v' \leq v$.  It suffices to prove that for any such $v'$, the localization $\psi^{v'}(s_\beta)$ is divisible by $\beta$, in order to prove that $\beta$ divides $\Psi^\varepsilon_{\gamma_{s_\beta }}$ and thus $\Psi^\varepsilon_{\gamma_{s_\beta w}}$.

Given any $\varepsilon' \in \{ 0,1\}^{\ell(v')}$ such that $(s_\beta)^{\varepsilon' }= v'$ and $|\varepsilon'| = \ell(v')$, there is a 1 in $\varepsilon'$ that is closest (on either side)  to the middle reflection $s_i$ in the palindromic expression $s_\beta = s_{u_1}\cdots s_{u_r} s_i s_{u_r}\cdots s_{u_1}$.  If the mask $\varepsilon'$ is supported on $s_i$ itself, then $\varepsilon'$ shows the hyperplane crossing $\beta$, in which case $\beta \mid \psi^{v'}(s_\beta)$.  Otherwise, there is support on some $s_{u_j}$ closest to $s_i$ in the expression.   Without loss of generality, assume that $s_{u_j}$ occurs in the lefthand portion of the palindrome. Denote by $\varepsilon'_j$ the entry of $\varepsilon'$ closest to but not supported on $s_i$, showing the reflection $s_{u_j}$ in the palindromic expression for $s_\beta$.  This corresponds to a unique mask $\varepsilon''\in \{ 0,1\}^{\ell(v')}$ which instead shows  $s_{u_j}$ in the righthand portion of the palindromic expression, and is otherwise identical to $\varepsilon'$. This mask also satisfies $(s_\beta)^{\varepsilon''} = v'$, and moreover, all masks $\varepsilon''$ for $\gamma_{s_\beta}$ such that $(s_\beta)^{\varepsilon''} = v'$ and $|\varepsilon''| = \ell(v')$ arise in this way.  

Thus, all remaining terms of $\psi^{v'}(s_\beta)$ come in pairs of the form $\Psi^{\varepsilon'}_{\gamma_{s_\beta}} = (\alpha') \prod_{1 = \varepsilon_{\ell \neq j}}  (\beta_{i_\ell}+k_{i_\ell}\delta)  $ and $\Psi^{\varepsilon''}_{\gamma_{s_\beta}} = (\alpha'') \prod_{1 = \varepsilon_{\ell \neq j} } (\beta_{i_\ell}+k_{i_\ell}\delta) $ for some $\alpha', \alpha'' \in R^+_{af}$, using the fact that $\varepsilon'$ and  $\varepsilon''$ agree everywhere except at the two crossings corresponding to the two $s_{u_j}$. Because we chose a palindromic expression for $s_\beta$, with a corresponding alcove walk symmetric over the hyperplane $H_{\beta}$, then up to a sign, $\alpha''$ is the reflection of $\alpha'$ over $H_\beta$; that is, $\pm \alpha'' = s_\beta(\alpha')$.  Now writing $\beta = \alpha+k\delta \in R_{af}^+$ and $\alpha' = \mu+m\delta \in R^+_{af}$, we apply formula \eqref{eq:affonaff} to obtain 
\[\pm \alpha''= s_{\alpha+k\delta}(\mu+m\delta) = \mu - \langle \alpha^\vee, \mu \rangle \alpha + \left( m- \langle k\alpha^\vee, \mu \rangle\right) \delta.
\]
Since the walk $\gamma_{s_\beta}$ is minimal, then since $\alpha' \in R_{af}^+$ corresponds to a step prior to the one crossing $\beta$. In fact $s_\beta(\alpha') \in R^-_{af}$, in which case $\alpha'' = -s_\beta(\alpha')$ so that this forward step is also indexed by a positive root.  We now compute the sum
\[ \alpha'+\alpha'' = \alpha' - s_\beta(\alpha') =  \langle \alpha^\vee, \mu \rangle \alpha +\langle k\alpha^\vee, \mu \rangle \delta = \langle \alpha^\vee, \mu\rangle ( \alpha + k\delta) = \langle \alpha^\vee, \mu \rangle \beta.
\]
Therefore, $\beta$ divides each sum $\Psi^{\varepsilon'}_{\gamma_{s_\beta}} + \Psi^{\varepsilon''}_{\gamma_{s_\beta}} = (\alpha' + \alpha'') \prod_{1 = \varepsilon_{\ell \neq j} } (\beta_{i_\ell}+k_{i_\ell}\delta) $.
Altogether,  $\beta \mid \psi^{v'}(s_\beta)$ and thus that $\beta \mid \Psi^{\varepsilon}_{\gamma_{s_\beta w}}$ for all masks $\varepsilon$ which show at least one crossing of $\gamma_{s_\beta}$.

Now consider those terms of $\psi^v(s_\beta w)$ indexed by masks supported only on $\gamma_w$. We show that the difference of these terms with $\psi^v(w)$ is divisible by $\beta$.  Given any mask $\varepsilon$ for $\gamma_w$ such that $w^\varepsilon = v$ and $|\varepsilon| = \ell(v)$, there is a corresponding mask $\varepsilon'$ for $\gamma_{s_\beta w}$ obtained by pre-appending $2r+1$ zeros to $\varepsilon$. Moreover, all masks $\varepsilon' \in \{0,1\}^{2r+m+1}$ such that $(s_\beta w)^{\varepsilon'} = v$ and $|\varepsilon'| = \ell(v)$ which do not show any crossings of $\gamma_{s_\beta}$ arise in this way. We again write $\beta = \alpha+k \delta \in R^+_{af}$.  Denote by $\alpha_\varepsilon = \mu+m\delta \in R_{af}$ the affine root corresponding to any crossing in $\gamma_w$ shown by $\varepsilon$. Since $\gamma_{s_\beta w}$ was constructed as the concatenation of $\gamma_{s_\beta}$ and $\gamma_w$, the corresponding positive root $\alpha_{\varepsilon'} \in R^+_{af}$ shown by $\varepsilon'$ equals $\alpha_{\varepsilon'} = s_\beta(\alpha_{\varepsilon})$.  

As in the previous case, we compute the difference $\alpha_\varepsilon - \alpha_{\varepsilon'} = \alpha_\varepsilon - s_\beta(\alpha_{\varepsilon}) =  \langle \alpha^\vee, \mu\rangle \beta,$ which shows that $\beta$ divides $\alpha_\varepsilon - \alpha_{\varepsilon'}$. Therefore, all terms of $\psi^v(w)$ pair with a unique term in $\psi^v(s_\beta w)$ indexed by a mask supported outside of $\gamma_{s_\beta}$, and for each such pair of masks $\varepsilon, \varepsilon'$ we have shown that the difference $\Psi^\varepsilon_{\gamma_w} -\Psi^{\varepsilon'}_{\gamma_{s_\beta w}}$ is divisible by $\beta$. The difference between $\psi^v(w)$ and the terms of $\psi^v(s_\beta w)$ indexed by masks supported only on $\gamma_w$ is thus divisible by $\beta$.  Having already shown that $\beta \mid \psi^v(s_\beta w)$ for terms indexed by masks with support on $\gamma_{s_\beta}$, we have that $\beta \ \vert \ \left( \psi^v(w) - \psi^v(s_\beta w) \right)$, as desired.
\end{proof}

\subsection{Localization in affine flag varieties}\label{sec:Local}

We are now prepared to state our second theorem, which provides a (non-minimal) alcove walk interpretation of the functions $\xi^v$ of Kostant and Kumar \cite{KostantKumar}.  In particular, not only do the functions $\psi^v$ satisfy the GKM condition as proved in Theorem \ref{thm:GKM}, but in fact they coincide precisely with the affine Schubert basis; see Figures \ref{fig:class_s_10} and \ref{fig:class_s_1} for examples of affine Schubert classes calculated via Theorem \ref{thm:Local}.

\begin{figure}
\begin{center}

	\begin{overpic}[scale=0.4]{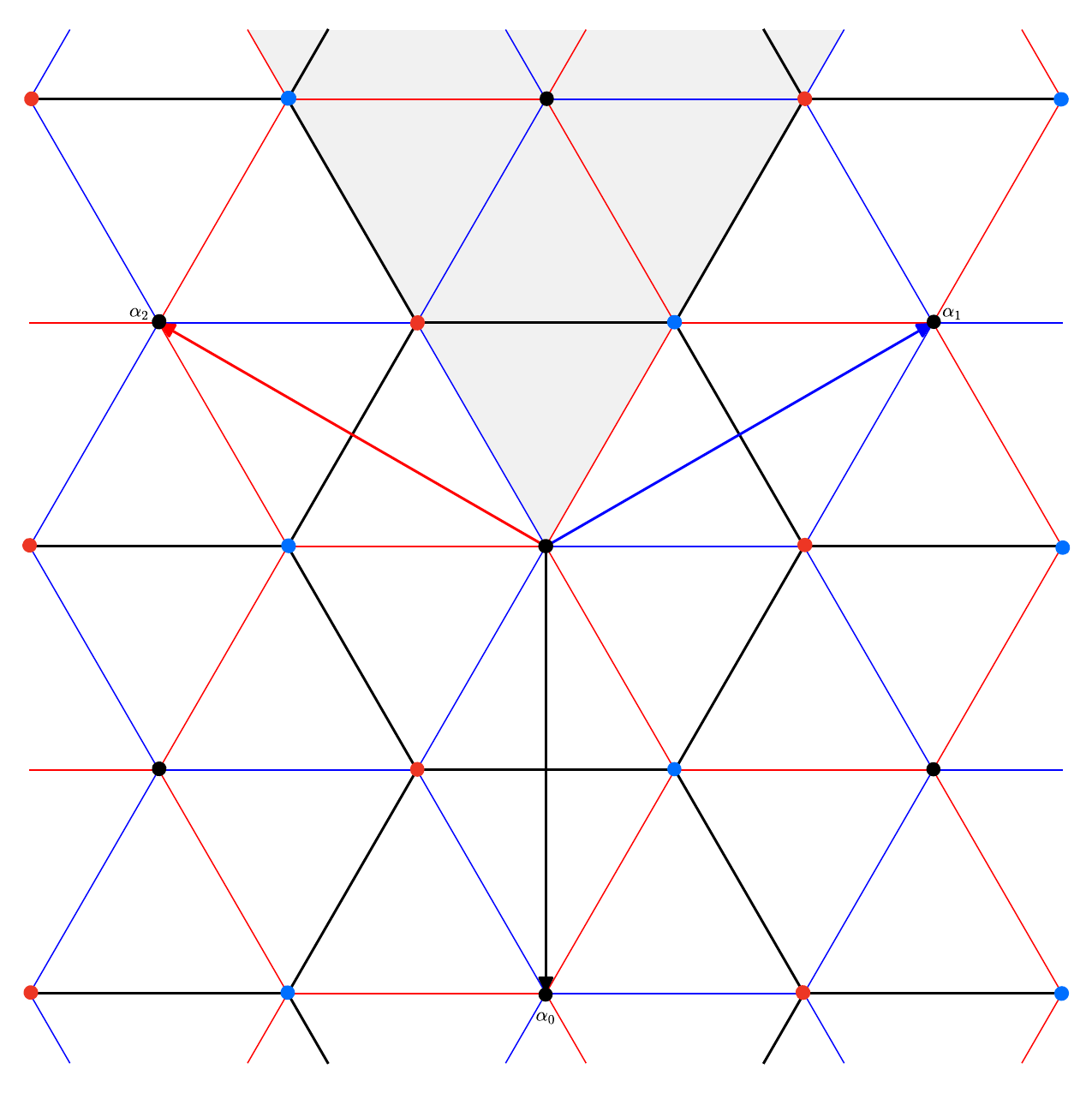}
	\put(49.5,55){{0}}
	\put(49.5,45){{0}}
	\put(45,47){{0}}
	\put(45,52){{0}}
	\put(54,47){{0}}
	\put(54,52){{0}}

	\put(49.5,95){{0}}
	\put(49.5,85){{0}}
	\put(45,87){{0}}
	\put(45,92){{0}}
	\put(54,87){{0}}
	\put(54,92){{0}}

	\put(44,21){\parbox{1.5cm}{$(\alpha_1+\alpha_2) \\ (\alpha_1+\alpha_2+\delta)$}} 
	\put(31,5){\parbox{1.5cm}{$(\alpha_1+\alpha_2) \\ (\alpha_1+\alpha_2+\delta)$}} 
	\put(31,12){\parbox{1.5cm}{$(\alpha_1+\alpha_2) \\ (\alpha_1+\alpha_2+\delta)$}} 
	\put(54,5){\parbox{1.5cm}{$(\alpha_1+\alpha_2) \\ (\alpha_1+\alpha_2+\delta)$}} 
	\put(54,12){\parbox{1.5cm}{$(\alpha_1+\alpha_2) \\ (\alpha_1+\alpha_2+\delta)$}} 

	\put(81,40){\parbox{1.5cm}{$(\alpha_1+\alpha_2) \\ (\alpha_2 + \delta)  $}} 
	\put(71,33){\parbox{1.5cm}{$(\alpha_1+\alpha_2) \\ (\alpha_2 + \delta) $}} 
	\put(71,25){\parbox{1.5cm}{$(\alpha_1+\alpha_2) \\ (\alpha_2 + \delta) $}} 
	\put(81,16){\parbox{1.5cm}{$(\alpha_1+\alpha_2) \\ (\alpha_2 + \delta) $}} 

	\put(10,40){\parbox{1.5cm}{$(\alpha_1) \\ (\alpha_1 + \delta)  $}} 
	\put(20,33){\parbox{1.5cm}{$(\alpha_1) \\ (\alpha_1 + \delta) $}} 
	\put(20,25){\parbox{1.5cm}{$(\alpha_1) \\ (\alpha_1 + \delta) $}} 
	\put(10,16){\parbox{1.5cm}{$(\alpha_1) \\ (\alpha_1 + \delta) $}} 

	\put(10,85){\parbox{1.5cm}{$(\alpha_1) \\ (-\alpha_2 + \delta)  $}} 
	\put(20,76){\parbox{1.5cm}{$(\alpha_1) \\ (-\alpha_2 + \delta) $}} 
	\put(20,66){\parbox{1.5cm}{$(\alpha_1) \\ (-\alpha_2 + \delta) $}} 
	\put(10,58){\parbox{1.5cm}{$(\alpha_1) \\ (-\alpha_2 + \delta) $}} 

	\put(81,85){\parbox{1.5cm}{0}} 
	\put(71,76){\parbox{1.5cm}{$0$}} 
	\put(71,66){\parbox{1.5cm}{$0 $}} 
	\put(81,58){\parbox{1.5cm}{$0$}} 

	\put(79,94){\parbox{2cm}{$(-\alpha_1+2\delta) \\ (-\alpha_1 -\alpha_2 + 2\delta) $}} 

	\end{overpic}
	\caption{The function $\psi^{s_{10}}$ in affine type $A_2$ evaluated at each alcove.}
	\label{fig:class_s_10}
	\end{center}
\end{figure}

\begin{theorem}\label{thm:Local}
Let $w,v \in W_{af}$. Then the values of the function $\xi^v$ are given by
\begin{equation*}
\xi^v(w) \; = \sum_{\substack{|\varepsilon|=\ell(v) \\ w^{\varepsilon} = v}}\Psi_\gamma^{\varepsilon} \ = \; \psi^v(w),
\end{equation*} 
where $\gamma$ is any (not necessarily minimal) alcove walk to $\w$.
As such, the function $\psi^v$ is the image of the affine Schubert class $[X_v]$ under the localization map \eqref{eq:AffLoc}.
\end{theorem}

\begin{proof}
For any $w,v \in W_{af}$, we must argue that $\xi^v(w) = \psi^v(w)$.  It suffices to prove that the polynomials $\psi^v(w) \in \ZZ[\alpha_0,\alpha_1, \dots, \alpha_{n-1}]$ satisfy both $\psi^v(e) = \delta_{v,e}$ where $\delta_{x,y}$ for $x,y \in W_{af}$ denotes the Kronecker delta function, as well as the following recursion for any $i \in I_{af}$
\begin{equation}\label{eq:recursion}
\psi^v(ws_i) = 
\begin{cases}
\psi^v(w) + (w\alpha_i)\psi^{vs_i}(w) & \text{if}\ vs_i<v \\
\psi^v(w) & \text{otherwise},
\end{cases}
\end{equation}
as these properties uniquely determine the function values $\xi^v(w)$ by Lemma 11.1.8 of \cite{Kumar}.  

When $w=e$, by Proposition \ref{prop:Indep} we can consider the trivial walk $\gamma_0 = (\base_0)$ to $\w = \base_0$.  For any $v \neq e$, there are no masks $\varepsilon$ such that $w^\varepsilon = v$, and so $\psi^v(e) = 0$. In case $v = e$ as well, the empty mask $\varepsilon_\emptyset$ satisfies $|\varepsilon_\emptyset| = \ell(v)$ and $w^{\varepsilon_\emptyset} = v$.  Since $\gamma_0$ does not cross any hyperplanes, we have $\Psi^{\varepsilon_\emptyset}_{\gamma_0} = \psi^e(e) = 1$ in this case.  Altogether, we thus have $\psi^v(e) = \delta_{v,e}$.

To prove the recursion \eqref{eq:recursion}, we split the argument into two cases.  First consider the case where $ws_i > w$.   We first fix a walk $\gamma$ of length $m$ to $\w$, which we may choose to be minimal by Proposition \ref{prop:Indep}, and we define the walk $\gamma'$ to be the same as $\gamma$, with one final step across the panel of $\w$ having type $i$.  Because $ws_i > w$, the walk $\gamma'$ remains minimal, and this final step from $\w$ to the alcove $ws_i \base_0$ is a forward step, labeled by the affine root $w\alpha_i \in R^+_{af}$.

To compute $\psi^v(ws_i)$ in case $ws_i > w$, we calculate $\Psi^\varepsilon_{\gamma'}$ for all masks $\varepsilon \in \{0,1\}^{m+1}$ such that $|\varepsilon| = \ell(v)$ and $(ws_i)^\varepsilon = v$.  If $vs_i>v$, then no reduced expression for $v$ ends in $s_i$, and therefore all masks such that both $|\varepsilon| = \ell(v)$ and  $(ws_i)^\varepsilon = v$ are fully supported on $\gamma$ and do not show this final hyperplane crossing in $\gamma'$.  In other words, if $vs_i > v$, then $\psi^v(ws_i) = \psi^v(w)$, as required by \eqref{eq:recursion}.

If instead $vs_i<v$, then there exist reduced expressions for $v$ which do end in $s_i$.  Grouping together all masks of minimal support such that $(ws_i)^\varepsilon = v$ which hide the final hyperplane crossing in $\gamma'$, we obtain $\psi^v(w)$ as before.  Now consider all masks $\varepsilon'$ of minimal support such that $(ws_i)^{\varepsilon'} = v$ which show the final hyperplane crossing in $\gamma'$.  As all such masks show this final crossing, the affine root $w\alpha_i$ is a factor in all summands $\Psi^{\varepsilon'}_{\gamma'}$. Otherwise, the mask $\varepsilon'$ is supported fully on $\gamma$ and corresponds to a mask $\overline{\varepsilon} \in \{0,1\}^m$ obtained by omitting the final 1 in the sequence $\varepsilon'$.  Moreover, since $vs_i<v$, then the mask $\overline{\varepsilon}$ satisfies both $|\overline{\varepsilon}| = \ell(vs_i) $ and $w^{\overline{\varepsilon}} = vs_i$.  In other words, 
\[ 
\psi^v(ws_i) =  \sum_{\substack{|\varepsilon|=\ell(v) \\ (ws_i)^{\varepsilon} = v}}\Psi_{\gamma'}^{\varepsilon} \ = \psi^v(w) +(w\alpha_i) \sum_{\substack{|\overline{\varepsilon}|=\ell(vs_i) \\ w^{\overline{\varepsilon}} = vs_i}}\Psi_\gamma^{\overline{\varepsilon}} \ = \psi^v(w) + (w\alpha_i)\psi^{vs_i}(w),
\]
verifying \eqref{eq:recursion} in case both $ws_i>w$ and $vs_i<v$.

It remains to consider the case where $ws_i< w$.  Fix a walk $\gamma$ of minimal length $m$ to $\w$ such that $\type(\gamma) = s_{i_1}\cdots s_{i_{m-1}}s_i$, meaning that the final crossing of $\gamma$ has type $i$.  To compute $\psi^v(ws_i)$, by Proposition \ref{prop:Indep} we may choose to define an alcove walk $\overline{\gamma}$ from $\base_0$ to $ws_i\base_0$ to be the same as $\gamma$, with the final crossing of type $i$ omitted.   

If $vs_i>v$, then no reduced expression for $v$ ends in $s_i$, and all masks $\overline{\varepsilon} \in \{0,1\}^{m-1}$ for $\overline{\gamma}$ such that $|\overline{\varepsilon}| = \ell(v)$ and $(ws_i)^{\overline{\varepsilon}} = v$ correspond to a unique mask $\varepsilon \in \{0,1\}^m$ for $\gamma$ obtained by adding a final 0 to the sequence $\overline{\varepsilon}$. As this final zero hides the last crossing, we again have $|\varepsilon| = \ell(v)$ and $w^\varepsilon = v$.  Moreover, all hyperplane crossings shown by $\overline{\varepsilon}$ and $\varepsilon$ coincide, and so $\psi^v(ws_i) = \psi^v(w)$ in this case.

Now consider the final case where $vs_i<v$.  Here we instead compute $\psi^v(w)$ by summing $\Psi^\varepsilon_\gamma$ over masks $\varepsilon \in \{0,1\}^m$ such that $|\varepsilon| = \ell(v)$ and $w^\varepsilon = v$.  Grouping together all such masks which hide the final hyperplane crossing in $\gamma$, we obtain $\psi^v(ws_i)$ using the correspondence  obtained by deleting the final entry of $\varepsilon$.  All other masks $\varepsilon' \in \{0,1\}^m$ of minimal support such that $w^{\varepsilon'} = v$ will show the final hyperplane crossing in $\gamma$.  Since $w = s_{i_1}\cdots s_{i_{m-1}}s_i$ is a reduced expression, then $s_{i_1}\cdots s_{i_{m-1}}(\alpha_i) \in R^+_{af}$ by \eqref{E:betaj}. Since $s_i\alpha_i = -\alpha_i$, then $-s_{i_1}\cdots s_{i_{m-1}}s_i(\alpha_i) = -w\alpha_i \in R^+_{af}$.  As the final crossing $\w\textbf{s}_i \to \w$ of $\gamma$ is a forward step, it is labeled by $-w\alpha_i \in R^+_{af}$.  As all remaining masks show this crossing, the affine root $-w\alpha_i$ is a factor in all summands of the form $\Psi^{\varepsilon'}_\gamma$.  Now taking $\varepsilon'$ and switching the final 1 to 0, we obtain a unique mask $\overline{\varepsilon} \in \{0,1\}^m$ such that $|\overline{\varepsilon}| = \ell(v)-1$ and $w^{\overline{\varepsilon}}=vs_i$.  In other words, the remaining masks $\varepsilon'$ which show the final crossing of $\gamma$ combine to produce $(-w \alpha_i)\psi^{vs_i}(w)$.  Putting these two collections of masks together, we have 
\[ \psi^v(w) = \psi^v(ws_i) - w\alpha_i\psi^{vs_i}(w), \]
verifying \eqref{eq:recursion} in the remaining case that both $ws_i<w$ and $vs_i < v$.

Finally, having proved that $\psi^v = \xi^v$, the fact that the function $\psi^v$ is the image of the affine Schubert class $[X_v]$ under the localization map \eqref{eq:AffLoc} now follows from Proposition 11.3.10 of \cite{Kumar}.
\end{proof}

We now provide several examples to illustrate Theorem \ref{thm:Local}, in addition to making some remarks which connect Theorem \ref{thm:Local} to related literature.

\begin{figure}
\begin{center}

	\begin{overpic}[scale=0.4]{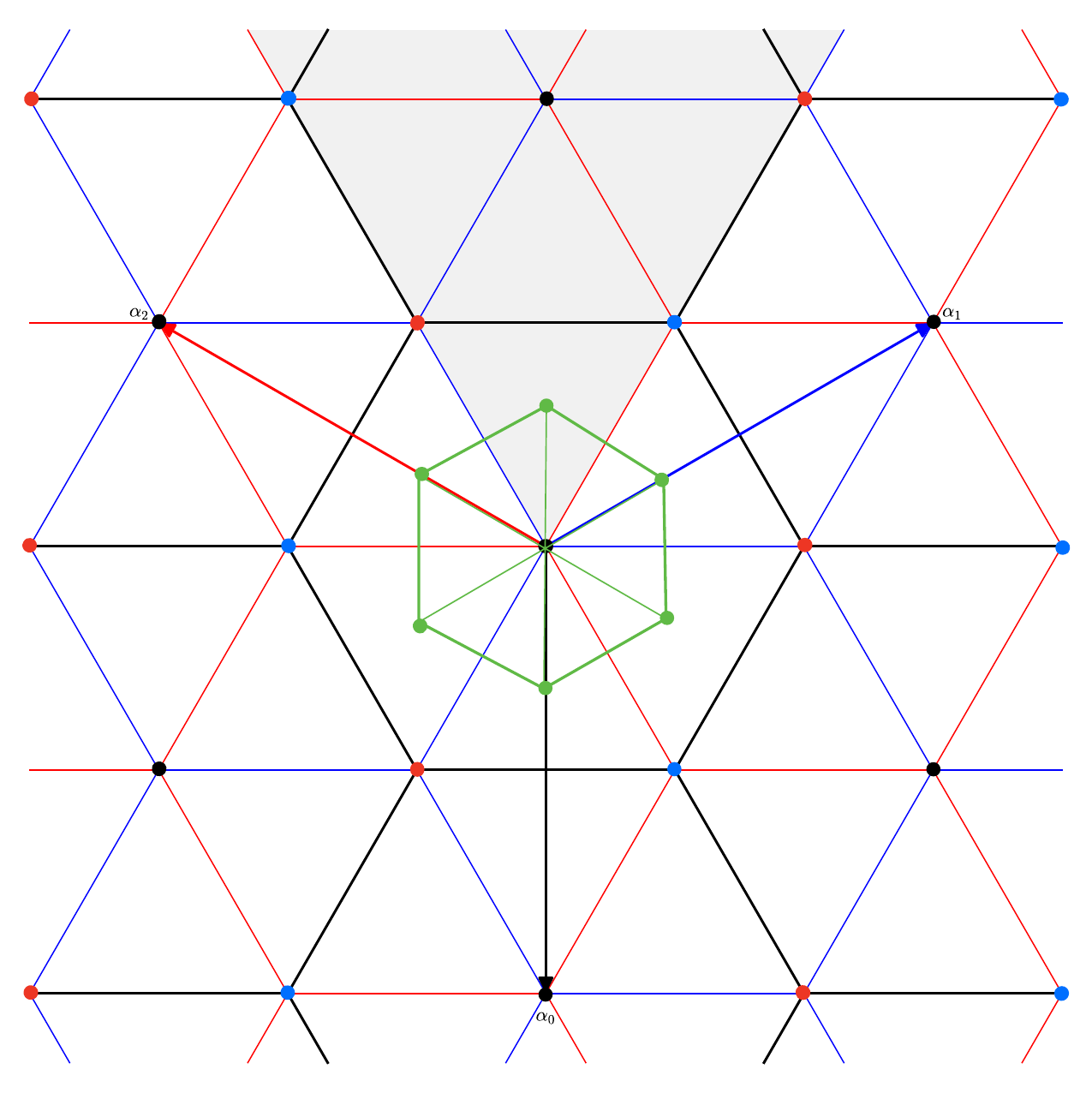}
	\put(61,62){{0}}
	\put(55,66){{0}}
	\put(55,73){{0}}
	\put(61,76){0}
	\put(67,66){{0}}
	\put(67,73){{0}}
	
	\put(58,38){{$\alpha_1+\alpha_2$}}
	\put(51,32){{$\alpha_1+\alpha_2$}}
	\put(51,26){{$\alpha_1+\alpha_2$}}
	\put(65,32){{$\alpha_1+\alpha_2$}}
	\put(65,26){{$\alpha_1+\alpha_2$}}
	\put(58,20){{$\alpha_1+\alpha_2$}}

	\put(32,46){{$\alpha_1$}}
	\put(32,53){{$\alpha_1$}}
	\put(18,46){{$\alpha_1$}}
	\put(18,53){{$\alpha_1$}}
	\put(25,57){{$\alpha_1$}}
	\put(25,42){{$\alpha_1$}}
	
	\put(30,88){{$-\alpha_2+\delta$}}
	\put(30,93){{$-\alpha_2+\delta$}}
	\put(14,88){{$-\alpha_2+\delta$}}
	\put(14,93){{$-\alpha_2+\delta$}}
	\put(22,81){{$-\alpha_2+\delta$}}

	\put(87,46){{$\alpha_2 + \delta$}}
	\put(87,53){{$\alpha_2 + \delta$}}
	
	\put(54,93){{$-\alpha_1 - \alpha_2 + 2\delta$}}
	
	\put(84,88){{$-\alpha_1 + 2\delta$}}
	\put(84,93){{$-\alpha_1 + 2\delta$}}

	\put(30,11){{$2\alpha_1 + \alpha_2 + \delta$}}
	\put(10,11){{$2\alpha_1 + \alpha_2 + \delta$}}
	\put(20,22){{$2\alpha_1 + \newline \alpha_2 + \delta$}}
	\put(30,6){{$2\alpha_1 + \alpha_2 + \delta$}}
	\put(10,6){{$2\alpha_1 + \alpha_2 + \delta$}}
	
	\put(81,11){{$\alpha_1 + 2\alpha_2 + \delta$}}
	\put(81,6){{$\alpha_1 + 2\alpha_2 + \delta$}}

	\end{overpic}
	\caption{The function $\psi^{s_1}$ in affine type $A_2$ evaluated at each alcove.}
	\label{fig:class_s_1}
	\end{center}
\end{figure}

\begin{example0}
The affine Schubert class $\psi^{s_1}$ obtained by applying Theorem \ref{thm:Local} is shown in Figure \ref{fig:class_s_1}.  We have labeled the values $\psi^{s_1}(w)$ for $w \in W_{af}$ to emphasize the symmetry around the blue vertices (not the origin).  We have also included in green a portion of the dual picture used in the traditional approach to GKM theory, obtained by placing a vertex in the center of each alcove, and connecting each vertex with an edge passing through any common face.  Labeling the six green vertices by the values shown in the corresponding alcove recovers the Schubert class $\sigma_{s_1}$ in the small-torus equivariant cohomology of $SL_3(\CC)/B$ depicted in Figure 1 of \cite{TymBilley}, illustrating how to recover classical results from Theorem \ref{thm:Local}.
\end{example0}

\begin{remark}
Theorem \ref{thm:Local} can also be viewed as providing a geometric interpretation of Billey's formula for calculating the localizations $\xi^v(w)$; see \cite{Billey} or the survey \cite{TymBilley} for explicit formulas. This connection is immediate by identifying the affine roots appearing in Billey's formula as inversions of $w \in W_{af}$ as in Lemma 1.3.14 of \cite{Kumar}, and then applying Proposition 11.1.11 of \cite{Kumar}.  Of course, proving that the functions $\psi^v$ coincide with $\xi^v$ could have thus proceeded by directly verifying that $\psi^v(w)$ reproduces Billey's formula.  

However, all previously known instances of Billey's formula require the initial choice of expression for $w$ to be minimal length, and this hypothesis is dropped in calculating $\psi^v(w)$ using the alcove walk method presented in this paper.  We have thus pursued the methods of proof that we believe to be best suited for the geometric packaging of alcove walks.  In particular, the technique of concatenating alcove walks as illustrated by Figure \ref{fig:sbetaw}, which regularly occurs in the proofs of Theorems \ref{thm:GKM} and \ref{thm:Local}, is not available in the context where a minimal walk to $\w$ is required.
\end{remark}

\begin{remark}
Functions satisfying the following three properties are called \emph{Knutson-Tao classes} in \cite{Tymoczko-AJM}, after the work of \cite{KnutsonTao}:
\begin{enumerate}
\item $\xi^v(w) = 0$ unless $v \leq w$,
\item $\xi^v(w)$ is homogeneous of degree $\ell(v)$,
\item $\xi^v(v) = \prod\limits_{\beta \in R^+_{af}, s_{\beta} v < v} \beta$. 
\end{enumerate} 
In case $v,w$ are elements of the finite Weyl group $W$, Lemma 2.16 of \cite{Tymoczko-AJM} says that any two Knutson-Tao classes indexed by $v \in W$ coincide.  As the Schubert basis for the finite flag variety is uniquely determined by these three properties, any system of Knutston-Tao classes thus coincides with the Schubert basis for $G/B$. 

It is clear by definition that the functions $\psi^v$ defined in terms of alcove walks satisfy all three of these properties; see Figures \ref{fig:class_s_10} and \ref{fig:class_s_1} for illustrations. However, the proofs in \cite{Tymoczko-AJM}, as well as similar results in \cite{GZ-Advances}, rely crucially on the finiteness of the graph.
To the best of our knowledge, a precise reference to an affine version of Lemma 2.16 from \cite{Tymoczko-AJM} is not available. 
See Proposition 4.3 of \cite{HHH} for a more general version, or Proposition 4 of \cite{LSkDouble} for a special case.    Based on a private communication \cite{LSprivate}, we believe that the proof of \cite[Prop.~4]{LSkDouble} goes through in the level of generality required to verify the requisite type-free, big-torus statement, and we encourage the interested reader to carry out this exercise.  
\end{remark}


\section{Labeled folded alcove walks and GKM theory}\label{sec:folded}

In this section we adapt the treatment of Parkinson, Ram, and Schwer in \cite{PRS} to provide a bijection in Theorem \ref{thm:AlcoveWalks} between the set of alcove walks which are folded according to certain rules and the points in the intersection of two different affine Schubert cells.  We conclude by discussing how such folded alcove walks are related to the summands of $\psi^v(w)$, which provide the localizations of the affine Schubert class $[X_v]$ at the $T_{af}$-fixed point indexed by $w$, according to Theorem \ref{thm:Local}.

\subsection{Positively folded alcove walks}\label{sec:positive}

 We now define an orientation on every affine hyperplane $H_\beta$, which separates $V^* \cong \RR^n$ into two halves.  Each side of $H_\beta$ is either positive or negative, as follows.
Given any $\beta \in R_{af}$, orient the hyperplane $H_{\beta}$ such that the base alcove $\base_0$ is on the \emph{positive} side of $H_\beta$.  The side of $H_\beta$ which is separated from $\base_0$ is defined to be the \emph{negative} side of $H_\beta$. The orientations for the three families of hyperplanes in affine type $A_2$ are indicated by $\pm$ in Figure \ref{fig:fold1}.
 In \cite{GraeberSchwer}, this is called the orientation induced by the base alcove $\base_0$.
 
 \begin{figure}[ht]
	\centering
	\begin{overpic}[scale = 0.35]{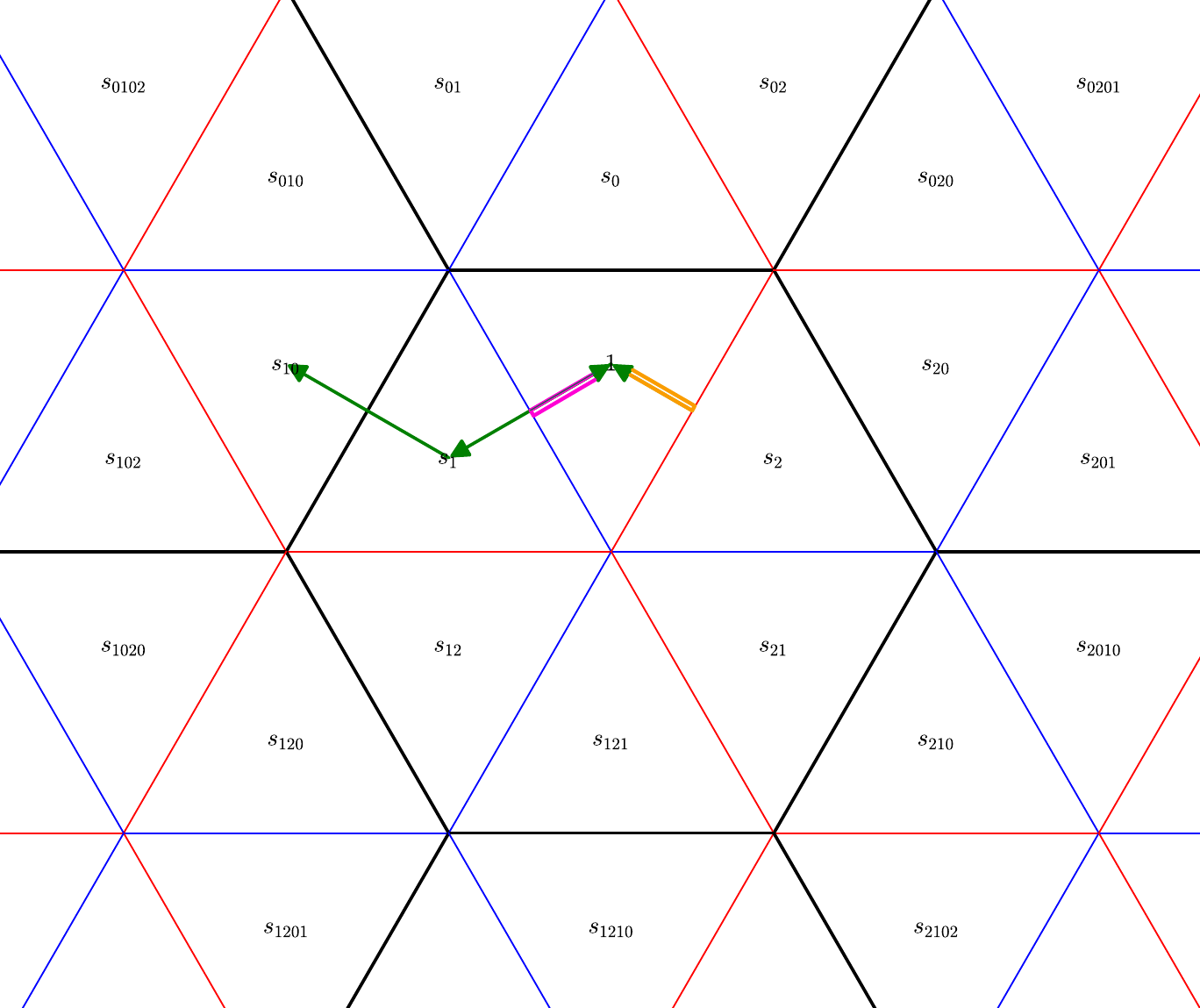}
	\put(74,76){{$\red{H_{\alpha_2 }}$}}
	\put(64,4){{$\blue{H_{\alpha_1 }}$}}
	\put(18,1){{\blue{$+$}}}
	\put(13.5,1){{\blue{$-$}}}		
	\put(45,1){{\blue{$+$}}}	
	\put(40.5,1){{\blue{$-$}}}	
	\put(73,1){{\blue{$+$}}}	
	\put(68.5,1){{\blue{$-$}}}	
	\put(96.5,1){{\blue{$+$}}}
	\put(96.5,15.5){{$+$}}		
	\put(96.5,11.5){{$-$}}	
	\put(96.5,39.5){{$+$}}		
	\put(96.5,35.5){{$-$}}	
	\put(96.5,62.5){{$-$}}		
	\put(96.5,58.5){{$+$}}				
	\put(22,80){{\red{$+$}}}
	\put(17.5,80){{\red{$-$}}}
	\put(50,80){{\red{$+$}}}
	\put(45.5,80){{\red{$-$}}}
	\put(77,80){{\red{$-$}}}
	\put(72.5,80){{\red{$+$}}}
	\put(98,80){{\red{$+$}}}				
	\end{overpic}
	\caption{The positively folded alcove walk of type $w = s_{1210}$ corresponding to the mask $(0,0,1,1)$ for $v = s_{10}$, having folds in the first (pink) and second (orange) steps.}
	\label{fig:fold1}
\end{figure}

Any orientation on hyperplanes then gives rise to an orientation on each step of an alcove walk. 
An alcove walk $\gamma = (\base_0, \base_1, \dots, \base_m)$ has a \emph{positive crossing} at step $j$ if $\gamma$ crosses from the negative side of the hyperplane $H_{\beta_j+k_j\delta}$ to the positive side, meaning that the alcove $\base_{j-1}$ is on the negative side of  $H_{\beta_j+k_j\delta}$ while $\base_j$ is on the positive side.  Conversely, $\gamma$ has a \emph{negative crossing} at step $j$ if $\gamma$ crosses from the positive side of $H_{\beta_j+k_j\delta}$ to the negative side. The two green steps in the alcove walk shown in Figure \ref{fig:fold1} are negative crossings.

For any $\alpha_j \in \Delta_{af}$ and $v \in W_{af}$, the notion of positivity on the affine roots relates to this orientation on hyperplane crossings as follows:
\begin{equation}\label{E:HRootLabelsI}
v \alpha_j \in R^+_{af} \quad \iff \quad \ell(v) < \ell(vs_j) \quad  \iff  \quad \textbf{v} \to \textbf{vs}_j \ \text{is negative.}
\end{equation}
Since $R_{af} = R^+_{af} \sqcup R^-_{af}$, then we also have
\begin{equation}\label{E:HRootLabelsI-}
v \alpha_j \in R^-_{af} \quad \iff \quad \ell(vs_j) < \ell(v) \quad  \iff  \quad \textbf{v} \to \textbf{vs}_j \ \text{is positive.}
\end{equation}

Given any alcove walk $\gamma = (\base_0, \base_1, \dots, \base_m)$, we say that $\gamma$ is \emph{folded at step $j$} if $\base_{j-1} = \base_j$.
Folded alcove walks also inherit an orientation from the relevant hyperplanes. 
If the alcove $\base_{j-1}$ is on the positive side of $H_{\beta_j+k_j \delta}$, we say that $\gamma$ is \emph{positively folded at step} $j$. We say that the alcove walk $\gamma$ is \emph{positively folded} if for all $1 \leq j \leq m$ such that $\gamma$ is folded at step $j$, then $\gamma$ is positively folded at step $j$.  In Figure \ref{fig:fold1}, both the first fold (in pink) and the second fold (in orange) occur on the positive side of their respective hyperplanes, and thus the alcove walk in Figure \ref{fig:fold1} is positively folded.

\subsection{Labeled alcove walks and affine Schubert cells}\label{sec:gpthy}

To connect to the group theory in the affine flag variety, as we do in this section, it will be necessary to assume that the underlying expression for $w = s_{i_1}\cdots s_{i_\ell} \in W_{af}$ is reduced throughout the remainder of the paper.

 For any $\beta \in R_{af}$, denote the corresponding one-parameter root subgroup by $\mathfrak{X}_{\beta} = \{ x_{\beta}(c) \mid c \in \CC\}$, and write $x_i(c) = x_{\alpha_i}(c)$ for the positive simple affine roots $\alpha_i \in \Delta_{af}$.  Note that $\mathfrak{X}_\beta$ is a subgroup of $I$ for $\beta \in R^+_{af}$, and that $\mathfrak{X}_\beta$ is a subgroup of $I^-$ for $\beta \in R^-_{af}$. Denote by $n_i = n_{\alpha_i}$ a lift to $G(F)$ of the simple reflection $s_i \in W_{af}$ for $i \in I_{af}$.  Following \cite[Eq.~3.3]{PRS}, for any $\beta \in R_{af}$ and $c \in \CC$, we define
 \begin{equation*}
 n_\beta(c) = x_\beta(c)x_{-\beta}(-c^{-1})x_\beta(c) \quad \text{and} \quad h_\beta(c) = n_\beta(c)n_\beta(1)^{-1}.
 \end{equation*}

\begin{figure}
	\centering
	\begin{overpic}[scale = 0.385]{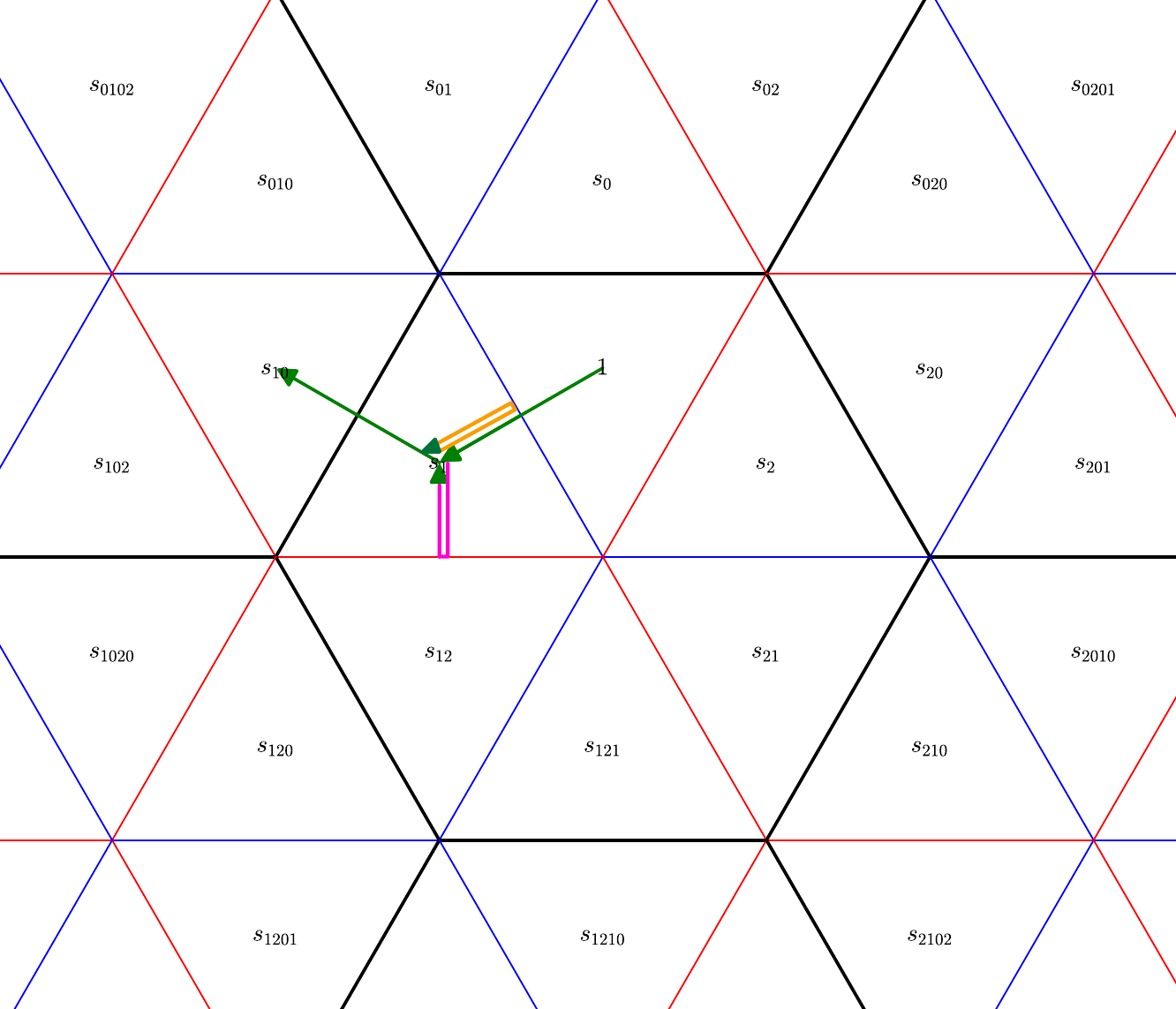}
	\put(87,40){{$H_{\alpha_1 + \alpha_2 }$}}. 
	\put(64.5,4){{$\blue{H_{\alpha_1 }}$}}
	\put(18,1){{\blue{$+$}}}
	\put(13.5,1){{\blue{$-$}}}		
	\put(45,1){{\blue{$+$}}}	
	\put(40.5,1){{\blue{$-$}}}	
	\put(73,1){{\blue{$+$}}}	
	\put(68.5,1){{\blue{$-$}}}	
	\put(96.5,1){{\blue{$+$}}}
	\put(96.5,15.5){{$+$}}		
	\put(96.5,11.5){{$-$}}	
	\put(96.5,39.5){{$+$}}		
	\put(96.5,35.5){{$-$}}	
	\put(96.5,62.5){{$-$}}		
	\put(96.5,58.5){{$+$}}				
	\put(22,80){{\red{$+$}}}
	\put(17.5,80){{\red{$-$}}}
	\put(50,80){{\red{$+$}}}
	\put(45.5,80){{\red{$-$}}}
	\put(77,80){{\red{$-$}}}
	\put(72.5,80){{\red{$+$}}}
	\put(98,80){{\red{$+$}}}	
	\put(45,49){{$0$}}
	\put(38,39.5){{$c_1 \in \CC^\times$}}
	\put(32.5,52){{$c_2 \in \CC^\times$}}
	\put(26,49){{$0$}}
	\end{overpic}
	\caption{The labeled folded alcove walk of type $w = s_{1210}$ corresponding to the mask $(1,0,0,1)$ for $v = s_{10}$, having folds in the second (pink) step labeled $c_1$ and the third (orange) step labeled $c_2$.}
	\label{fig:fold2}
\end{figure}

There are two convenient ways to represent those points in the affine flag variety $G/I$ which lie in the double coset $IwI$.  Given any reduced expression $w = s_{i_1}\cdots s_{i_\ell}$, we have
\begin{equation}\label{E:IwIpoints1}
IwI = \left\{ x_{i_1}(c_1)n_{i_1}^{-1}x_{i_2}(c_2)n_{i_2}^{-1}\cdots x_{i_\ell}(c_\ell)n_{i_\ell}^{-1}I \ \middle|\ c_j \in \CC \right\};
\end{equation}
see \cite[Lemma 7.4]{Ronan}.
Recall that we can also associate this reduced expression for $w$ with the alcove walk $\gamma_w = (\base_0, \base_1, \dots, \base_\ell)$ such that $\type(\gamma_w) = s_{i_1}\cdots s_{i_\ell}$.
Considering \eqref{E:IwIpoints1}, the points of $G/I$ which lie in $IwI$ are in bijection with choices for the field elements $c_i \in \CC$.  We view these parameters as \emph{labels} on the crossings of the alcove walk $\gamma_w$ as follows:
\begin{itemize}
\item If $\gamma$ has a positive crossing at step $j$, this step is labeled by an element of $\CC$.
\item If $\gamma$ has a negative crossing at step $j$, this step is labeled by $0$.
\item If $\gamma$ has a positive fold at step $j$, then this step is labeled by an element of $\CC^\times$.
\end{itemize}
We refer to $\gamma_w$ as a \emph{labeled} alcove walk.  The corresponding labels on the folded alcove walk in Figure \ref{fig:fold2} are shown.

An alternative way to write the points of $G/I$ which lie in $IwI$ is given by \cite[Theorem 4.1]{PRS}, which says that
\begin{equation}\label{E:IwIpoints2}
IwI = \{ x_{\beta_1}(c_1)x_{\beta_2}(c_2)\cdots x_{\beta_\ell}(c_\ell)n_wI \mid c_j \in \CC \},
\end{equation}
where $n_w = n_{i_1}^{-1} \cdots n_{i_\ell}^{-1}$, and the roots $\beta_j \in R^+_{af}$ are the elements of $\inv(w)$ defined in \eqref{E:betaj}; see also Theorem 15 and Lemma 43 in \cite{Steinberg}.

We now recall three \emph{Steinberg relations} on the group $G(F)$, which provide the means for rewriting the points in $IwI$ as points in a double coset of the form $I^-vI$ for some $v\in W_{af}$. 

\begin{enumerate}
\item For any $\beta \in R_{af}$ and $c \in \CC^\times$,  we can write
\begin{equation}\label{E:MainFold}
x_{\beta}(c)n_\beta^{-1} = x_{-\beta}(c^{-1})x_\beta(-c)h_{\beta}(c).
\end{equation}
\item For any $\alpha, \beta \in R_{af}$ and $c \in \CC$, we have the conjugation relation
\begin{equation}\label{E:WConj}
n_\alpha x_\beta(c)n_\alpha^{-1} = x_{s_\alpha(\beta)}(\pm c),
\end{equation}
where the sign $\pm c$ is uniquely determined by the pair $\alpha$ and $\beta$.
\item If $b \in I$, since $bx_i(c)n_i^{-1} \in Is_iI$, then by \eqref{E:IwIpoints1} there exist unique $\tilde{c} \in \CC$ and $b' \in I$ such that
\begin{equation}\label{E:bMove}
bx_i(c)n_i^{-1} = x_i(\tilde{c})n_i^{-1}b'.
\end{equation}
\end{enumerate}
See Equations (7.6), (3.6), and (7.7) in \cite{PRS}, respectively, and the references therein, as well as \cite{BilligDyer}.

\subsection{Labeled folded alcove walks and affine Richardson varieties}

We now define the relevant set of alcove walks to identify points in open affine Richardson varieties in Theorem \ref{thm:AlcoveWalks} below.
\begin{definition}
Let $w,v \in W_{af}$.  Denote by $\mathcal{P}(\w, v)$ the set of all minimal, labeled alcove walks of type $\vec{w}$ which are positively folded and end in alcove $\mathbf{v} = v\base_0$. 
\end{definition}
Note that $\mathcal{P}(\w, v)$ restricts to the set of minimal alcove walks of type $\w$, in contrast with previous sections which did not require the alcove walk to be minimal.

We are now ready to state and prove our final theorem, which gives an alcove walk interpretation for the points in an open affine Richardson variety. 
\begin{theorem}\label{thm:AlcoveWalks}
Let $w,v \in W_{af}$.  Then there is a bijection 
$$ (IwI \cap I^-vI)/I  \quad \longleftrightarrow \quad \mathcal{P}(\vec{w},v).$$  
\end{theorem}
The proof of Theorem \ref{thm:AlcoveWalks} follows \cite[Theorem 7.1]{PRS}, which treats the context of unipotent orbits and periodic orientations on the hyperplanes.  This proof of Theorem \ref{thm:AlcoveWalks} inspired the argument for the more general  \cite[Theorem 5.10]{MNST}. A similar statement which requires alternate techniques is \cite[Theorem 1.2]{MST-IIG}.

\begin{proof}
The proof proceeds by induction on the length of $w$.  For the base case, suppose that $\ell(w) = 1$ so that $w = s_j$ for some $j \in I_{af}$.   By \eqref{E:IwIpoints1}, we can write $Is_jI = \{ x_j(c)n_j^{-1}I \mid c \in \CC\}$. 
 If $c=0$, then 
\begin{equation*}
x_j(c)n_j^{-1}I = x_{\alpha_j}(0)n_j^{-1}I = x_{-\alpha_j}(0)n_j^{-1}I \in I^-s_jI.
\end{equation*}
In particular, $Is_jI \subseteq I^-s_jI$, and so $Is_jI \cap I^- v I \neq \emptyset$ if and only if $v=s_j$ by the affine Bruhat decomposition \eqref{eq:bruhat}.
Moreover, $(Is_jI \cap I^- s_jI)/I$ contains a single point in $G/I$, which is in natural bijection with the unique alcove walk $\base_0 \to \mathbf{s}_j$ of type $s_j$ having a single negative crossing at $H_{\alpha_j}$ labeled by $c=0$.

If instead $c \neq 0$, then we apply \eqref{E:MainFold} to write 
\begin{equation*}
x_j(c)n_j^{-1}I = x_{-\alpha_j}(c^{-1})x_{\alpha_j}(-c)h_{\alpha_j}(c)I = x_{-\alpha_j}(c^{-1})I \in I^- 1 I,
\end{equation*}
and so $Is_jI \cap I^- v I \neq \emptyset$ if and only if $v=1$.  Moreover, the points in $(Is_jI \cap I^- 1 I)/I$ are parameterized by the element $c^{-1} \in \CC^{\times}$, which naturally corresponds to the folded alcove walk $\base_0 \to \base_0$ of type $s_j$ having a single positive fold at $H_{\alpha_j}$ labeled by the element $c^{-1} \in \CC^{\times}$.  Altogether, we have thus shown that $(Is_jI \cap I^-vI)/I  \longleftrightarrow \mathcal{P}(\vec{s}_j,v)$, where these sets are nonempty if and only if $v \in \{ 1, s_j\}$, establishing the base case.

For the inductive step, suppose that $\ell(w) >1$, and choose a reduced expression $w = s_{i_1} \cdots s_{i_\ell}$. By the inductive hypothesis, assume that we have a bijection $(IwI \cap I^-vI)/I  \longleftrightarrow \mathcal{P}(\vec{w},v)$.  For each $v \in W_{af}$ such that these sets are nonempty, we may use \eqref{E:IwIpoints1} and \eqref{E:IwIpoints2} to write
\begin{equation}\label{E:GP3eq1}
x_{i_1}(c_1)n_{i_1}^{-1} \cdots x_{i_\ell}(c_\ell)n_{i_\ell}^{-1} = x_{\beta_1}(d_1)\cdots x_{\beta_\ell}(d_\ell)n_vb \in I^- v I
\end{equation}
for some $n_v = n_{j_1}^{-1}\cdots n_{j_k}^{-1}$ with $v = s_{j_1} \cdots s_{j_k}$ and $\beta_1, \dots, \beta_\ell \in R^-_{af}$. 

Now for any $j \in I_{af}$ such that $\ell(ws_j) > \ell(w)$, we must show that  
\[ (Iws_jI \cap I^-yI)/I  \longleftrightarrow \mathcal{P}(\vec{ws}_j,y).\]
Since $ws_j = s_{i_1} \cdots s_{i_\ell}s_j$ is reduced, all elements of $\mathcal{P}(\vec{ws}_j,y)$ can be obtained by concatenating each element of $\mathcal{P}(\vec{w},v)$ with one final crossing $\textbf{v} \to \textbf{vs}_j$, and then considering all possible folds and labels applicable to this additional crossing. 
 On the other hand, the points in $Iws_jI \cap I^-yI$ are parameterized by right multiplying by the element $x_j(c)n_j^{-1}$ in expression \eqref{E:GP3eq1}.  Using \eqref{E:bMove}, we then write
\begin{align}
x_{i_1}(c_1)n_{i_1}^{-1} \cdots x_{i_\ell}(c_\ell)n_{i_\ell}^{-1} x_j(c)n_j^{-1} & = x_{\beta_1}(d_1)\cdots x_{\beta_\ell}(d_\ell)n_vbx_j(c)n_j^{-1}  \nonumber \\
& =  x_{\beta_1}(d_1)\cdots x_{\beta_\ell}(d_\ell)n_vx_j(\tilde{c})n_j^{-1}b' \label{E:indstep}
\end{align}
for unique $\tilde{c} \in \CC$ and $b' \in I$.  We must now consider several cases.
 
 (1) First suppose that the crossing $\textbf{v} \to \textbf{vs}_j$ is positive.
Use \eqref{E:WConj} to rewrite \eqref{E:indstep} as
\begin{align*}
x_{\beta_1}(d_1)\cdots x_{\beta_\ell}(d_\ell)n_vx_j(\tilde{c})n_j^{-1}b' & = x_{\beta_1}(d_1)\cdots x_{\beta_\ell}(d_\ell)n_vx_j(\tilde{c})n_v^{-1}n_vn_j^{-1}b' \\
& = x_{\beta_1}(d_1)\cdots x_{\beta_\ell}(d_\ell)x_{v\alpha_j}(\pm \tilde{c})n_vn_j^{-1}b' \\
& = x_{\beta_1}(d_1)\cdots x_{\beta_\ell}(d_\ell)x_{v\alpha_j}(\pm \tilde{c})n_{vs_j}b'.
\end{align*}
 Since $\textbf{v} \to \textbf{vs}_j$ is positive, then  $v\alpha_j \in R^-_{af}$ by \eqref{E:HRootLabelsI-}, and so $x_{\beta_1}(d_1)\cdots x_{\beta_\ell}(d_\ell)x_{v\alpha_j}(\pm \tilde{c}) \linebreak \in I^-$. 
  In particular, $I ws_j I \cap I^- y I \neq \emptyset$ if and only if $y = vs_j$ in this case, and the points in $(I ws_j I \cap I^- vs_j I)/I $ are parameterized by the element $\pm \tilde{c} \in \CC$. The points in this intersection are thus in natural bijection with those labeled folded alcove walks from $\base_0$ to $\mathbf{vs}_j$ of type $\vec{ws}_j$, obtained from the corresponding element of $\mathcal{P}(\vec{w}, v)$ by concatenating a final positive crossing at $H_{v\alpha_j}$ labeled by the element $\pm \tilde{c}$.

 (2) Now suppose that the crossing $\textbf{v} \to \textbf{vs}_j$ is negative, and further suppose that $\tilde{c} \neq 0$.  Since $\tilde{c} \neq 0$, we can use \eqref{E:MainFold} and \eqref{E:WConj} to rewrite \eqref{E:indstep} as
\begin{align*}
x_{\beta_1}(d_1)\cdots x_{\beta_\ell}(d_\ell)n_vx_j(\tilde{c})n_j^{-1}b' & = 
x_{\beta_1}(d_1)\cdots x_{\beta_\ell}(d_\ell)n_vx_{-\alpha_j}(\tilde{c}^{-1})x_{\alpha_j}(-\tilde{c})h_{\alpha_j}(\tilde{c})b'  
\\
& = x_{\beta_1}(d_1)\cdots x_{\beta_\ell}(d_\ell)n_vx_{-\alpha_j}(\tilde{c}^{-1})b'' \\
& = x_{\beta_1}(d_1)\cdots x_{\beta_\ell}(d_\ell)n_vx_{-\alpha_j}(\tilde{c}^{-1})n_v^{-1}n_vb'' \\
& = x_{\beta_1}(d_1)\cdots x_{\beta_\ell}(d_\ell)x_{-v\alpha_j}(\pm \tilde{c}^{-1})n_vb''
\end{align*}
for some $b'' \in I$.  Since $\textbf{v} \to \textbf{vs}_j$ is negative, then  $v\alpha_j \in R^+_{af}$ by \eqref{E:HRootLabelsI},  in which case $-v\alpha_j \in R^-_{af}$.  Therefore, $x_{\beta_1}(d_1)\cdots x_{\beta_\ell}(d_\ell)x_{-v\alpha_j}(\pm \tilde{c}^{-1}) \in I^-$. 
 In particular, $Iws_jI \cap I^- y I \neq \emptyset$ if and only if $y = v$ in this case, and the points in $( Iws_jI \cap I^- v I)/I$ are parameterized by the element $\pm \tilde{c}^{-1} \in \CC^{\times}$.  The points in this intersection are thus in natural bijection with those labeled folded alcove walks from $\base_0$ to $\mathbf{v}$ of type $\vec{ws}_j$, obtained from the corresponding element of $\mathcal{P}(\vec{w}, v)$ by folding the final negative crossing at $H_{v\alpha_j}$ labeled by the element $\pm \tilde{c}^{-1}$.

(3) Now suppose that the crossing $\textbf{v} \to \textbf{vs}_j$ is negative, and finally suppose that $\tilde{c} = 0$.  Use \eqref{E:WConj}  to rewrite \eqref{E:indstep} as
\begin{align*}
x_{\beta_1}(d_1)\cdots x_{\beta_\ell}(d_\ell)n_vx_{\alpha_j}(0)n_j^{-1}b' & = 
x_{\beta_1}(d_1)\cdots x_{\beta_\ell}(d_\ell)n_vx_{-\alpha_j}(0)n_j^{-1}b'\\
& = x_{\beta_1}(d_1)\cdots x_{\beta_\ell}(d_\ell)n_vx_{-\alpha_j}(0)n_v^{-1}n_vn_j^{-1}b'\\
& = x_{\beta_1}(d_1)\cdots x_{\beta_\ell}(d_\ell)x_{-v\alpha_j}(0)n_vn_j^{-1}b'\\
& = x_{\beta_1}(d_1)\cdots x_{\beta_\ell}(d_\ell)x_{-v\alpha_j}(0)n_{vs_j}b'.
\end{align*}
As in the previous case, $x_{\beta_1}(d_1)\cdots x_{\beta_\ell}(d_\ell)x_{-v\alpha_j}(0) \in I^-$, and  $Iws_jI \cap I^-yI \neq \emptyset$ if and only if $y = vs_j$ in this case.  The points in $(Iws_jI \cap I^-vs_jI)/I$ are then clearly in bijection with those labeled folded alcove walks from $\base_0$ to $\mathbf{vs}_j$ of type $\vec{ws}_j$, obtained from the corresponding element of $\mathcal{P}(\vec{w},v)$ by concatenating a final negative crossing at $H_{v\alpha_j}$ labeled by $\tilde{c}=0$.

Altogether, we have shown that $(I ws_j I \cap I^- y I)/I \longleftrightarrow \mathcal{P}(\vec{ws}_j,y)$, where these sets are nonempty if and only if $y \in \{ v, vs_j\}$, and the result follows by induction.
\end{proof}

\subsection{Labeled folded alcove walks and GKM theory}

In this final section, we explain how the alcove walks indexing the summands of the localizations $\psi^v(w)$ and their corresponding masks can be used to produce folded alcove walks.

\begin{definition}
Let $\gamma = (\base_0, \base_1, \dots, \base_m)$ be an alcove walk of $\type(\gamma) = s_{i_1}\cdots s_{i_m}$. We can \emph{fold} $\gamma$ at step $j$ by defining a new alcove walk
\[ \gamma' = (\base_0, \dots, \base_{j-1}, \base'_j, \dots, \base'_m), \]
where $\base'_k = s_{i_{1}}\cdots \widehat{s_{i_j}} \cdots s_{i_{k}}\base_0$ for all $k$ such that $j \leq k \leq m$.  Equivalently, to obtain $\gamma'$, we repeat the alcove $\base'_{j} = \base_{j-1}$, and the entire walk $\gamma$ after step $j$ is then reflected across the hyperplane $H_{\beta_j+k_j\delta}$ crossed by $\gamma$ at step $j$. 
\end{definition}

The process of folding alcove walks can be iterated as follows.
\begin{definition}\label{def:foldedimage}
Let $\gamma = (\base_0, \base_1, \dots, \base_m)$ be a minimal alcove walk.  Let $\varepsilon \in \{0,1\}^m$ be a mask, and record the zero entries hiding crossings of $\gamma$ as $\{\varepsilon_{j_1}, \dots, \varepsilon_{j_k}\}$ where $j_1 < \cdots < j_k$.  First fold $\gamma$ at step $j_1$ to obtain $\gamma_1$.  Repeat this process for all $\ell$ such that $2 \leq \ell \leq k$ by folding $\gamma_{\ell-1}$ at step $j_{\ell}$ to obtain $\gamma_{\ell}$.  The resulting folded alcove walk $\gamma_k$ is called the \emph{$\varepsilon$-folded image of $\gamma$}, denoted $\gamma_\varepsilon$.
\end{definition} 

The folded alcove walks we have seen in this section illustrate this construction.

\begin{example0}\label{ex:foldwalks}
The folded walks in Figures \ref{fig:fold1} and \ref{fig:fold2} are the $\varepsilon$-folded images of the green walk $\gamma$ from Figure \ref{fig:braid}, with respect to the masks $\varepsilon = (0,0,1,1)$ and $\varepsilon = (1,0,0,1)$, respectively.  
The walk in Figure \ref{fig:fold1} is positively folded, as both the pink and orange folds occur on hyperplanes facing the base alcove $\base_0$. However, the walk in \ref{fig:fold2} is \emph{not} positively folded, as the orange fold occurs on the negative side of $H_{\alpha_1}$.
\end{example0}

While the folded alcove walk in Figure \ref{fig:fold2} is not positive with respect to the orientation defined in Section \ref{sec:positive}, it is positively folded with respect to the \emph{trivial positive} orientation from \cite{GraeberSchwer}, which labels both sides of every affine hyperplane by $+$.  Interestingly, by Proposition 6.8 of \cite{GraeberSchwer} (see also Lemma 2.2.1 in \cite{BjornerBrenti}), the alcove walk in Figure \ref{fig:fold2}, which is only positively folded with respect to the trivial positive orientation, can be algorithmically exchanged for the walk in Figure \ref{fig:fold1}, which is an honest element of $\mathcal{P}(\w, v)$ for $\w = s_{1210}$ and $v = s_{10}$.

In general, all masks $\varepsilon$ indexing summands of $\psi^v(w)$ have $\varepsilon$-folded images which are positively folded with respect to the trivial positive orientation, but only some masks give rise to elements of $\mathcal{P}(\w,v)$ which are positively folded with respect to the orientation used to establish Theorem \ref{thm:AlcoveWalks}.  We invite the interested reader to explore the precise characterization of the subset of masks whose folded images correspond to the elements of $\mathcal{P}(\w,v)$.


\begin{thebibliography}{GKM2}

\bibitem[1]{AJS}
H.~H. Andersen, J.~C. Jantzen, and W.~Soergel.
\newblock Representations of quantum groups at a {$p$}th root of unity and of
  semisimple groups in characteristic {$p$}: independence of {$p$}.
\newblock {\em Ast\'{e}risque}, (220):321, 1994.

\bibitem[2]{BjornerBrenti}
Anders Bj\"{o}rner and Francesco Brenti.
\newblock {\em Combinatorics of {C}oxeter groups}, volume 231 of {\em Graduate
  Texts in Mathematics}.
\newblock Springer, New York, 2005.

\bibitem[3]{BilligDyer}
Yuly Billig and Matthew Dyer.
\newblock Decompositions of {B}ruhat type for the {K}ac-{M}oody groups.
\newblock {\em Nova J. Algebra Geom.}, 3(1):11--39, 1994.

\bibitem[4]{Billey}
Sara~C. Billey.
\newblock Kostant polynomials and the cohomology ring for {$G/B$}.
\newblock {\em Duke Math. J.}, 96(1):205--224, 1999.

\bibitem[5]{billera1989grobner}
Louis~J Billera and Lauren~L Rose.
\newblock Gr{\"o}bner basis methods for multivariate splines.
\newblock In {\em Mathematical methods in computer aided geometric design},
  pages 93--104. Elsevier, 1989.

\bibitem[6]{BilleyWarrington}
Sara~C. Billey and Gregory~S. Warrington.
\newblock Kazhdan-{L}usztig polynomials for 321-hexagon-avoiding permutations.
\newblock {\em J. Algebraic Combin.}, 13(2):111--136, 2001.

\bibitem[7]{GKM}
Mark Goresky, Robert Kottwitz, and Robert MacPherson.
\newblock Equivariant cohomology, {K}oszul duality, and the localization
  theorem.
\newblock {\em Invent. Math.}, 131(1):25--83, 1998.

\bibitem[8]{GKMAffSpringer}
Mark Goresky, Robert Kottwitz, and Robert Macpherson.
\newblock Homology of affine {S}pringer fibers in the unramified case.
\newblock {\em Duke Math. J.}, 121(3):509--561, 2004.

\bibitem[9]{GraeberSchwer}
Marius Graeber and Petra Schwer.
\newblock Shadows in {C}oxeter groups.
\newblock {\em Ann. Comb.}, 24(1):119--147, 2020.

\bibitem[10]{goldin2009towards}
Rebecca~F Goldin and Susan Tolman.
\newblock Towards generalizing schubert calculus in the symplectic category.
\newblock {\em Journal of Symplectic Geometry}, 7(4):449--473, 2009.

\bibitem[11]{gilbert2016generalized}
Simcha Gilbert, Julianna Tymoczko, and Shira Viel.
\newblock Generalized splines on arbitrary graphs.
\newblock {\em Pacific Journal of Mathematics}, 281(2):333--364, 2016.

\bibitem[12]{guillemin2002combinatorial}
Victor Guillemin and Catalin Zara.
\newblock Combinatorial formulas for products of thom classes.
\newblock {\em Geometry, mechanics, and dynamics}, pages 363--405, 2002.

\bibitem[13]{GZ-Advances}
Victor Guillemin and Catalin Zara.
\newblock The existence of generating families for the cohomology ring of a
  graph.
\newblock {\em Adv. Math.}, 174(1):115--153, 2003.

\bibitem[14]{HHH}
Megumi Harada, Andr\'{e} Henriques, and Tara~S. Holm.
\newblock Computation of generalized equivariant cohomologies of {K}ac-{M}oody
  flag varieties.
\newblock {\em Adv. Math.}, 197(1):198--221, 2005.

\bibitem[15]{KostantKumar}
Bertram Kostant and Shrawan Kumar.
\newblock The nil {H}ecke ring and cohomology of {$G/P$} for a {K}ac-{M}oody
  group {$G$}.
\newblock {\em Adv. in Math.}, 62(3):187--237, 1986.

\bibitem[16]{KnutsonTao}
Allen Knutson and Terence Tao.
\newblock Puzzles and (equivariant) cohomology of {G}rassmannians.
\newblock {\em Duke Math. J.}, 119(2):221--260, 2003.

\bibitem[17]{Kumar}
Shrawan Kumar.
\newblock {\em Kac-{M}oody groups, their flag varieties and representation
  theory}, volume 204 of {\em Progress in Mathematics}.
\newblock Birkh\"{a}user Boston, Inc., Boston, MA, 2002.

\bibitem[18]{kSchur}
Thomas Lam, Luc Lapointe, Jennifer Morse, Anne Schilling, Mark Shimozono, and
  Mike Zabrocki.
\newblock {\em {$k$}-{S}chur functions and affine {S}chubert calculus},
  volume~33 of {\em Fields Institute Monographs}.
\newblock Springer, New York; Fields Institute for Research in Mathematical
  Sciences, Toronto, ON, 2014.

\bibitem[19]{lai2007spline}
Ming-Jun Lai and Larry~L Schumaker.
\newblock {\em Spline functions on triangulations}.
\newblock Number 110. Cambridge University Press, 2007.

\bibitem[20]{LSkDouble}
Thomas Lam and Mark Shimozono.
\newblock {$k$}-double {S}chur functions and equivariant (co)homology of the
  affine {G}rassmannian.
\newblock {\em Math. Ann.}, 356(4):1379--1404, 2013.

\bibitem[21]{LSprivate}
Thomas Lam and Mark Shimozono.
\newblock private communication.
\newblock 2023.

\bibitem[22]{Matsumoto}
Hideya Matsumoto.
\newblock G\'{e}n\'{e}rateurs et relations des groupes de {W}eyl
  g\'{e}n\'{e}ralis\'{e}s.
\newblock {\em C. R. Acad. Sci. Paris}, 258:3419--3422, 1964.

\bibitem[23]{MNST}
Elizabeth Mili{\'c}evi{\'c}, Yusra Naqvi, Petra Schwer, and Anne Thomas.
\newblock A gallery model for affine flag varieties via chimney retractions.
\newblock \textit{Transform. Groups}, DOI: 10.1007/s00031-022-09726-8, 2022.

\bibitem[24]{MST-IIG}
Elizabeth Mili{\'c}evi{\'c}, Petra Schwer, and Anne Thomas.
\newblock A gallery model for affine flag varieties via chimney retractions.
\newblock \textit{Innov. Incidence Geom.} to appear, arXiv.org/pdf/2207.12923,
  2022.

\bibitem[25]{PRS}
James Parkinson, Arun Ram, and Christoph Schwer.
\newblock Combinatorics in affine flag varieties.
\newblock {\em J. Algebra}, 321(11):3469--3493, 2009.

\bibitem[26]{Ram}
Arun Ram.
\newblock Alcove walks, {H}ecke algebras, spherical functions, crystals and
  column strict tableaux.
\newblock {\em Pure Appl. Math. Q.}, 2(4, Special Issue: In honor of Robert D.
  MacPherson. Part 2):963--1013, 2006.

\bibitem[27]{Ronan}
Mark Ronan.
\newblock {\em Lectures on buildings}.
\newblock University of Chicago Press, Chicago, IL, 2009.
\newblock Updated and revised.

\bibitem[28]{SageMath}
{Sage Developers}.
\newblock {\em {S}ageMath, the {S}age {M}athematics {S}oftware {S}ystem
  ({V}ersion 22.04)}, 2023.
\newblock {\tt https://www.sagemath.org}.

\bibitem[29]{Sage-combinat}
The {S}age-{C}ombinat community.
\newblock {S}age-{C}ombinat: enhancing {S}age as a toolbox for computer
  exploration in algebraic combinatorics, 2008.
\newblock {\tt http://combinat.sagemath.org}.

\bibitem[30]{schenck1997spectral}
Hal Schenck.
\newblock A spectral sequence for splines.
\newblock {\em Advances in Applied Mathematics}, 19(2):183--199, 1997.

\bibitem[31]{schenck1997local}
Hal Schenck and Mike Stillman.
\newblock Local cohomology of bivariate splines.
\newblock {\em Journal of Pure and Applied Algebra}, 117:535--548, 1997.

\bibitem[32]{Steinberg}
Robert Steinberg.
\newblock {\em Lectures on {C}hevalley groups}.
\newblock Yale University, New Haven, Conn., 1968.
\newblock Notes prepared by John Faulkner and Robert Wilson.

\bibitem[33]{Tits}
Jacques Tits.
\newblock Le probl\`eme des mots dans les groupes de {C}oxeter.
\newblock In {\em Symposia {M}athematica ({INDAM}, {R}ome, 1967/68), {V}ol. 1},
  pages 175--185. Academic Press, London, 1969.

\bibitem[34]{TymBilley}
Julianna Tymoczko.
\newblock Billey's formula in combinatorics, geometry, and topology.
\newblock In {\em Schubert calculus---{O}saka 2012}, volume~71 of {\em Adv.
  Stud. Pure Math.}, pages 499--518. Math. Soc. Japan, [Tokyo], 2016.

\bibitem[35]{TymGKM}
Julianna~S. Tymoczko.
\newblock An introduction to equivariant cohomology and homology, following
  {G}oresky, {K}ottwitz, and {M}ac{P}herson.
\newblock In {\em Snowbird lectures in algebraic geometry}, volume 388 of {\em
  Contemp. Math.}, pages 169--188. Amer. Math. Soc., Providence, RI, 2005.

\bibitem[36]{Tymoczko-AJM}
Julianna~S. Tymoczko.
\newblock Permutation representations on {S}chubert varieties.
\newblock {\em Amer. J. Math.}, 130(5):1171--1194, 2008.

\end{thebibliography}

\newcommand{\etalchar}[1]{$^{#1}$}

\end{document}